\documentclass[11pt]{amsart}

\usepackage{amsmath}
\usepackage{tikz}
\usepackage{amsfonts}
\usepackage{amssymb}
\usepackage{amsthm}
\usepackage{mathdots}
\usepackage{graphicx}
\graphicspath{{ipe_images/}}

\usepackage{hyperref}
\usepackage[capitalize,nameinlink,noabbrev,nosort]{cleveref}
\usepackage{comment}
\crefname{figure}{Figure}{Figures}
\usepackage{float}
\usepackage{caption}
\usepackage{subcaption}

\numberwithin{equation}{section}

\newtheorem{theorem}{Theorem}[section]
\newtheorem{lemma}[theorem]{Lemma}
\newtheorem{corollary}[theorem]{Corollary}
\newtheorem{definition}[theorem]{Definition}
\newtheorem{example}[theorem]{Example}
\newtheorem{proposition}[theorem]{Proposition}
\newtheorem{remark}[theorem]{Remark}

\crefname{theorem}{Theorem}{Theorems}
\crefname{proposition}{Proposition}{Propositions}
\crefname{corollary}{Corollary}{Corollaries}
\crefname{example}{Example}{Examples}
\crefname{definition}{Definition}{Definitions}
\crefname{lemma}{Lemma}{Lemmas}
\crefname{section}{Section}{Sections}
\DeclareMathOperator{\Pf}{Pf}
\DeclareMathOperator{\sgn}{sgn}
\newcommand{\ds}{\displaystyle}

\title{The monopole-dimer model on Cartesian products of plane graphs}
\author{Anita Arora}
\address{Anita Arora, Department of Mathematics, Indian Institute of Science, Bangalore  560012, India.}
\email{anitaarora@iisc.ac.in}

\author{Arvind Ayyer}
\address{Arvind Ayyer, Department of Mathematics, Indian Institute of Science, Bangalore  560012, India.}
\email{arvind@iisc.ac.in}

\date{\today}

\begin{document}

\begin{abstract}
The monopole-dimer model is a signed variant of the monomer-dimer model which has determinantal structure.
We extend the monopole-dimer model for planar graphs 
 (\emph{Math. Phys. Anal. Geom.}, 2015) to Cartesian products thereof and show that the partition function of this model can be expressed as a determinant of a generalised signed adjacency matrix.
We then show that the partition function is independent of the orientations of the planar graphs so long as the orientations are Pfaffian.
When these planar graphs are bipartite, we show that the computation of the partition function becomes especially simple.
We then give an explicit product formula for the partition function of three-dimensional grid graphs a la Kasteleyn and Temperley--Fischer, which turns out to be fourth power of a polynomial when all grid lengths are even. Finally, we generalise this product formula to $d$ dimensions, again obtaining an explicit product formula.
We conclude with a discussion on asymptotic formulas for the free energy and monopole densities. 
\end{abstract}


\maketitle

\section{Introduction}
The \emph{dimer model} originally arose as the study of the physical process of adsorption of diatomic molecules (like oxygen) on the surface of a solid. Abstractly it can be thought of as enumerating perfect matchings in an edge-weighted graph. For planar graphs, Kasteleyn~\cite{Kasteleyn1963} solved the problem completely by showing that the partition function can be written as a Pfaffian of a certain adjacency matrix built using a special class of orientations called Pfaffian orientations on the graph. An immediate corollary of Kasteleyn's result is that the Pfaffian is independent of the orientation.
For the case of two-dimensional grid graphs $Q_{m,n}$, Kasteleyn~\cite{KASTELEYN19611209}
and Temperley--Fisher \cite{Fisher,TemperleyFisher} independently gave an explicit product formula. For example, when $m$ and $n$ are even,
horizontal (resp. vertical) edges have weight $a$ (resp. $b$), 
 the partition function can be written as
\begin{equation}
\label{eqn:dimer-pf}
2^{mn/2} \prod_{i=1}^{m/2} \prod_{j=1}^{n/2}
\left( a^2 \cos^2 \frac{i \pi}{m+1} + b^2 \cos^2 \frac{j \pi}{n+1}
\right).
\end{equation}
This formula is remarkable because although each factor is a degree-two polynomial in $a$ and $b$ with not-necessarily rational coefficients, the product turns out to be a polynomial with nonnegative integer coefficients. In particular, when $a = b = 1$, it is not obvious from this formula that the resulting product is an integer.

There have been attempts to generalise the dimer model while preserving this nice structure. The natural physical generalisation is the \emph{monomer-dimer model}, which represents adsorption of a gas cloud consisting of both monoatomic and diatomic molecules. 
The abstract version here is the {(weighted)} enumeration of all matchings of a graph. {The weights are interpreted as energies and are positive real numbers.}
This is known to be a computationally difficult problem~\cite{jerrum1987} and the partition function here does not have such a clean formula. However, when there is a single monomer on the boundary of a plane graph, the partition function can indeed be written as a Pfaffian~\cite{wu2006}.
A lower bound for the partition function of the monomer-dimer model for $d$-dimensional grid graphs has been obtained by Hammersley--Menon~\cite{hammersley-menon-1970} by generalising the method of Kasteleyn and Temperley--Fisher.

In another direction, a signed version of the monomer-dimer model called the \emph{monopole-dimer model} has been 
introduced~\cite{Ayyer2015ASM} for planar graphs. Configurations of the monopole-dimer model can be thought of as superpositions of two monomer-dimer configurations having monomers (called monopoles there) at the same locations. Thus, one ends up with even loops and isolated vertices. What makes the monopole-dimer model less physical is that configurations have a signed weight {and they cannot be interpreted as energies anymore}. On the other hand, the partition function here can be expressed as a determinant. Moreover, it is a perfect square for a $2m \times 2n$ grid graph. A combinatorial interpretation of the square root is given in~\cite{ayyer-2020}.

{In \cite{Ayyer2015ASM}, a more general model called the \emph{loop-vertex model} has also been defined for a general graph together with an orientation. The partition function in this case can also be written as a determinant. However, this model depends on the orientation. One of the main motivations for this work is to find natural families of non-planar graphs where the partition function is independent of the orientation, just as in the monopole-dimer model. 
The second motivation comes from the intuition that the monopole-dimer model is an `integrable variant' of the more physical monomer-dimer model. If this is correct,  asymptotic properties of both models should be similar.
This has been explained in \cite{Ayyer2015ASM} for two dimensional grid. We expect this to hold for high-dimensional grids also. This is not easy to see because of the signs in monopole-dimer weights. We hope that our work will be a starting point towards establishing this relationship between the two models.
We note in passing that higher dimensional dimer models have started attracting attention; see \cite{CSW-2023,HLT-2023} for example.
}

We formulate the monopole-dimer model for Cartesian products of plane graphs in \cref{sec:MainRe}.
A key ingredient in the formulation is the construction of special directed cycle decompositions of certain projections, which are themselves plane graphs with parallel edges.
We first show {in \cref{thm:DetFor2}} that the partition function is a determinant of a generalised adjacency matrix built using Pfaffian orientations. As in the dimer model, we see immediately {in \cref{cor:kProdisOrinInd}} that the determinant is independent of the orientation. In \cref{sec:MDMbipartite case}, we focus attention on the monopole-dimer model on the Cartesian product of bipartite plane graphs. 
Here, we will show in {\cref{thm:IndCycDecom}} that we can allow arbitrary cycle decompositions of the projections mentioned above. This seems to be a new observation of independent interest.

We then focus on the special family of grid graphs in higher dimensions. We give an explicit product formula for the partition function of the monopole-dimer model on three-dimensional grid graphs in {\cref{thm:PF3Dgrid}} {generalising the expression} \eqref{eqn:dimer-pf}. One peculiar feature of this partition function is that it is a fourth power of a polynomial when all side lengths are even. Just as for the partition function of the monopole-dimer model for two-dimensional grids, it would be interesting to obtain a combinatorial interpretation of the fourth root.
We then briefly discuss the higher dimensional case in \cref{sec:HDgridGraph} and give a similar explicit product formula {in \cref{thm:PFkDgrid}}. We will also discuss its asymptotic behaviour in \cref{sec:asym}.

We begin with the background definitions and previous results.

\section{Dimer Model}
\label{sec:MMDM}
We begin by recalling basic terminology from graph theory.
A {\emph{(simple) graph}} is an ordered pair $G = (V(G), E(G))$, where $V(G)$ is the set of \emph{vertices} of $G$ and $E(G)$ is a collection of two-element subsets of $V(G)$, known as \emph{edges}.
{When we allow multiple edges between a pair of vertices (also called \emph{parallel edges}), we will call such objects \emph{multigraphs}.
We will never allow {self loops}.}
We will work with undirected graphs and we will always assume that the graphs are finite and {naturally vertex-labeled} from $\{1,2,\dots,|V(G)|\}$.  {A simple graph is therefore a graph with no parallel edges. All our models will be defined on simple graphs. We will encounter graphs with parallel edges only in certain decompositions.} 
Recall that a \textit{planar} graph is a graph which can be embedded in the plane, i.e. it can be drawn in such a way that no edges will cross each other. Such an embedding of a planar graph is referred as a \textit{plane graph} and it divides the whole plane into regions, each of which is called a \textit{face}. We will consider only those embeddings of the graph for which parallel edges do not enclose any vertex.
An \textit{orientation} on a graph $G$ is the assignment of arrows to its edges. A graph $G$ with an orientation $\mathcal{O}$ is called an \textit{oriented graph} and is denoted $(G,\mathcal{O})$. An orientation on a labeled graph obtained by orienting its edges from lower to higher labeled vertex is called a \textit{canonical orientation}. 

\begin{definition}
An orientation on a plane graph $G$ is said to be \emph{Pfaffian} if it satisfies the property that each simple loop enclosing a bounded face has an odd number of clockwise oriented edges.
A Pfaffian orientation is said to possess the \textit{clockwise-odd property}. 
\end{definition}

\begin{figure}[h!]
    \centering
    \includegraphics[scale=1]{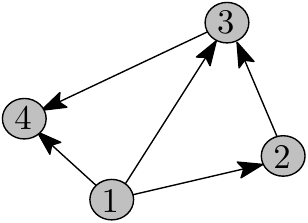}
    \caption{An oriented graph on $4$ vertices.}
    \label{fig:my_label8}
\end{figure}

For example, the orientation in \cref{fig:my_label8} is a Pfaffian orientation. Kasteleyn has shown that every plane graph has a Pfaffian orientation~\cite{Kasteleyn1963}. A \textit{dimer covering} or \emph{perfect matching} is a collection of edges in the graph $G$ such that each vertex is covered in exactly one edge. 
The set of all dimer coverings of $G$ will be denoted as $\mathcal{M}(G)$.
Let $G$ be an edge-weighted graph on $2n$ vertices with {real positive} weight $w_e$ for $e \in E(G)$ ({thought of as the energy of $e$}). Then the \textit{dimer model} is the collection of all dimer coverings where the weight of each dimer covering $M\in \mathcal{M}(G)$ is given by $w(M)=\prod_{e \in M}w_e$.
The \textit{partition function} of the dimer model on $G$ is then defined as
\begin{equation*}
    \mathrm{Z}_G:=\sum_{M \in \mathcal{M}(G)} w(M).
\end{equation*}
To state Kasteleyn's celebrated result, recall that {a matrix $A=(a_{i,j})$ is skew-symmetric if $a_{i,j}=-a_{j,i}$ for every $i,j$, and }the \textit{Pfaffian} of $2n \times 2n$ skew-symmetric matrix $A$ is given by
\begin{equation*}
\Pf(A)=\frac{1}{2^n n!} \sum_{\sigma \in  S_{2n}} \sgn(\sigma)\, A_{\sigma_1,\sigma_2} A_{\sigma_3,\sigma_4} \dots A_{\sigma_{2n-1},\sigma_{2n}},
\end{equation*}
and Cayley's theorem~\cite{Cayley1849} says that for such a matrix, ${\det A} = \Pf(A)^2$.

\begin{theorem}[Kasteleyn~\cite{Kasteleyn1963}] 
If $G$ is a plane graph with Pfaffian orientation $\mathcal{O}$, then the partition function of the dimer model on $G$ is given by $\mathrm{Z}_G=\Pf(K_G)$,
where $K_G$ is a signed adjacency matrix defined by
\begin{align*}
    (K_G)_{u,v}&=\begin{cases}
        w_e &  u\rightarrow v \text{ in }\mathcal{O},\\
      -w_e & v\rightarrow u \text{ in } \mathcal{O},\\
        \,0 &\text{otherwise.}
    \end{cases}
\end{align*}
\end{theorem}

{Throughout this article, we will refer to cycles in configurations on graphs as \emph{loops}. We will always understand these loops to be directed.}
Let us now recall generalization of the dimer model known as the loop-vertex model~\cite{Ayyer2015ASM}. Let $G$ be a simple weighted graph on $n$ vertices with an orientation $\mathcal{O}$, vertex-weights $x(v)$ for $v\in V(G)$ and edge-weights $a_{v,v'}\equiv a_{v',v}$ for $(v,v') \in E(G)$. A
 \textit{loop-vertex configuration} $C$ of $G$
is a subgraph of $G$ consisting of 
\begin{itemize}
    \item directed loops of even length (with length at least four),
    \item doubled edges (which can be thought of as loops of length two), and
    \item isolated vertices,
    \end{itemize}
with the condition that each vertex of $G$ is either an isolated vertex or is covered in exactly one loop. The set of all loop-vertex configurations of $G$ will be denoted as $\mathcal{L}(G)$. \cref{fig:Example of conf} shows a graph and two loop-vertex configurations on it.

\begin{figure}[h!]
\centering
\begin{subfigure}{.32\textwidth}
  \centering
  \includegraphics[width=0.6\linewidth]{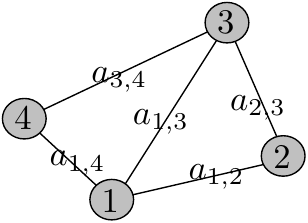}
  \caption{The graph of \cref{fig:my_label8}}
  \label{fig:sub1}
\end{subfigure}
\begin{subfigure}{.32\textwidth}
  \centering
  \includegraphics[width=0.6\linewidth]{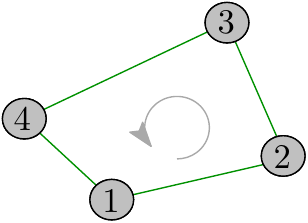}
  \caption{The directed cycle $(1234)$}
  \label{fig:sub2}
\end{subfigure}
\begin{subfigure}{.32\textwidth}
  \centering
  \includegraphics[width=0.6\linewidth]{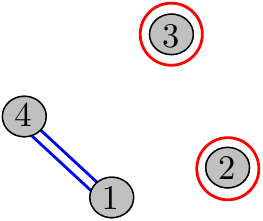}
  \caption{The doubled edge $(14)$ and \\ isolated vertices 2 and 3}
  \label{fig:sub3}
\end{subfigure}
\caption{The graph of \cref{fig:my_label8} with edge weights marked in \cref{fig:sub1}, and two loop-vertex configurations on it in \cref{fig:sub2} and \cref{fig:sub3}.}
\label{fig:Example of conf}
\end{figure}

The \textit{sign} of an edge $(v,v') \in E(G)$, is given by 
\begin{equation}
\label{eqn:SE}
   \sgn(v,v'):=\begin{cases}
    1 &  v \rightarrow v' \text{ in } \mathcal{O},
   \\ -1 &  v' \rightarrow v \text{ in } \mathcal{O}.
    \end{cases}
\end{equation}
Let $\ell=(v_0,v_1,\dots,v_{2k-1},v_{2k}=v_0)$ be a directed even loop in $G$. The 
\textit{weight }of the loop $\ell$ is given by
\begin{equation}
\label{eqn:weightofloop}
w(\ell):=-\prod_{i=0}^{2k-1} \sgn(v_i,v_{i+1})\, a_{v_i,v_{i+1}}. 
\end{equation}
Note that the weight of a doubled edge $(v_i,v_j)$ is always $+a_{v_i,v_j}^2$.

Then the \textit{loop-vertex model} on the pair $(G,\mathcal{O})$ is the collection $\mathcal{L}(G)$ with the weight of a configuration, $C=(\ell_1,\dots,\ell_j; \,v_1,\dots,v_k)$ consisting of loops $\ell_1,\dots,\ell_j$ and isolated vertices $v_1,\dots,v_k$, given by
    \begin{equation}
    \label{eqn:weightofconf}
            w(C)=\prod_{i=1}^j w(\ell_i) \,\,\prod_{i=1}^k x(v_i).
    \end{equation}
The \textit{(signed) partition function} of the loop-vertex model is defined as
\begin{equation*}
    \mathcal{Z}_{G,\mathcal{O}}:=\sum_{C \in \mathcal{L}(G)}w(C).
\end{equation*}

\begin{example}
\label{eg:PFLVM}
Let $G$ be a weighted graph on four vertices with vertex weights $x$ for all the vertices and edge weights as shown in \cref{fig:sub1}. Then the weights of the configuration shown in \cref{fig:sub2,fig:sub3} are $a_{1,2}a_{2,3}a_{3,4}a_{1,4}$ and $x^2a_{1,4}^2$. The partition function of the loop-vertex model on the graph in \cref{fig:sub1} with canonical orientation is 
\[
\mathcal{Z}_{G,\mathcal{O}} = x^4+a_{1,2}^2x^2+a_{1,3}^2x^2+a_{1,4}^2x^2+a_{2,3}^2x^2+a_{3,4}^2x^2+a_{1,2}^2a_{3,4}^2+a_{1,4}^2a_{2,3}^2+2a_{1,2}a_{2,3}a_{3,4}a_{1,4}.
\]
\end{example}

\begin{theorem}\cite[Theorem~2.5]{Ayyer2015ASM}
\label{thm:pf-lvmodel}
The partition function of the loop-vertex model on $(G,\mathcal{O})$ is 
\begin{equation*}
    \mathcal{Z}_{G,\mathcal{O}}={\det{\mathcal{K}_G}},
\end{equation*}
where $\mathcal{K}_G$ is a generalised adjacency matrix of $(G,\mathcal{O})$ defined as
\begin{equation}
\label{eqn:SAM}
    \mathcal{K}_G(v,v')=\begin{cases}
    \,x(v) &  v=v',\\
    \, a_{v,v'} &   v\rightarrow v'\, in\, \mathcal{O},\\
    -a_{v,v'} &   v'\rightarrow v\, in\, \mathcal{O},\\
    \quad 0 &   (v,v') \notin E(G).
    \end{cases}
\end{equation}
\end{theorem}

\begin{example}
The generalised adjacency matrix for the graph $G$ in \cref{eg:PFLVM} with the canonical orientation is
\[
\mathcal{K}_G=\begin{pmatrix}
x & a_{1,2} & a_{1,3}& a_{1,4}\\
- a_{1,2}&x & a_{2,3}& 0\\
- a_{1,3}& -a_{2,3}& x & a_{3,4}\\
- a_{1,4}& 0& -a_{3,4} & x
\end{pmatrix},
\]
and 
\[
\det \mathcal{K}_G = 
x^4+a_{1,2}^2x^2+a_{1,3}^2x^2+a_{1,4}^2x^2+a_{2,3}^2x^2+a_{3,4}^2x^2+a_{1,2}^2a_{3,4}^2+a_{1,4}^2a_{2,3}^2+2a_{1,2}a_{2,3}a_{3,4}a_{1,4},
\]
which is exactly $\mathcal{Z}_{G,\mathcal{O}}$ from \cref{eg:PFLVM}.
\end{example}

If $G$ is a simple vertex- and edge-weighted plane graph and $\mathcal{O}$ is a Pfaffian orientation on it, then the loop-vertex model is called the \textit{monopole-dimer model}. In that case, it has been shown~\cite[Theorem 3.3]{Ayyer2015ASM} that the 
weight of a loop $\ell=(v_0,v_1,\dots,v_{2k-1},v_{2k}=v_0)$ can be written as
    \begin{equation}
    \label{eqn:MDMEQ}
        w(\ell)=(-1)^{\text{number of vertices enclosed by $\ell$}} \,\prod_{j=0}^{2k-1}  a_{v_j,v_{j+1}}.
    \end{equation}
Then \cref{thm:pf-lvmodel} shows that the determinant of the generalised adjacency matrix of a plane graph with a Pfaffian orientation is independent of the latter.

\section{Monopole-dimer model on Cartesian products of plane graphs} 
\label{sec:MainRe}
We now extend the definition of the monopole-dimer model to Cartesian products of plane graphs.
The \textit{Cartesian product of two graphs} $G_1$ and $G_2$ is the graph denoted $G_1 \square G_2$ with vertex set $V(G_1)\times V(G_2)$ and edge set  
\begin{equation*}
\left\{((u_1,u_2),(u'_1,u'_2)) \middle\vert
\begin{aligned}
\text{either } u_1 = u'_1 \text{ and } (u_2,u'_2) \in E(G_2)  \\ 
\text{or } u_2 = u'_2 \text{ and } (u_1,u'_1) \in E(G_1) 
\end{aligned}
\right\}.
\end{equation*}
The above definition generalises to the \textit{Cartesian product of $k$ graphs} $G_1,\dots, G_k$, denoted $G_1 \square \allowbreak\cdots \square G_k$.
We will denote edges in $G_1 \square \cdots \square G_k$ of the form $((u_1,\dots,u_i,\dots,u_k), (u_1,\dots,u_i^\prime,\allowbreak \dots,u_k))$ \textit{$G_i$-edges}. Recall that a \emph{path graph} is a simple graph whose vertices can be arranged in a (non-repeating) linear sequence in such a way that two vertices are adjacent if and only if they are consecutive in the sequence. Clearly, a path graph is plane.
Let $P_n$ denote the \textit{path graph} on $n$ vertices.
It is clear from the definition that the Cartesian product of $k$ path graphs is a cuboid in $\mathbb{Z}^k$, also known as a \emph{grid graph}.
\cref{fig:LabelledP4P3} shows the Cartesian product $P_4 \square P_3$.
We will use the notation $[n]$ for the set $\{1,\dots,n\}$.

      \begin{figure}[h!]
        \centering
        \includegraphics[width=0.28\linewidth]{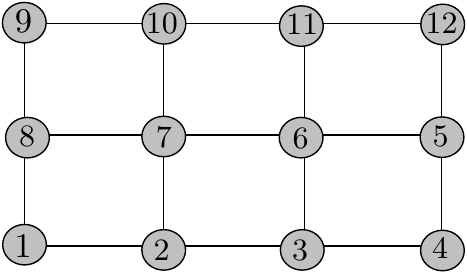}
        \caption{The Cartesian product $P_4\square P_3$ with its boustrophedon labelling; see \cref{sec:MDM3D}.}
        \label{fig:LabelledP4P3}
    \end{figure}

The \emph{degree} of a vertex is the number of edges incident to it and
an \textit{even} graph $G$ is one in which all the vertices have even degree. 
A \emph{walk} in a graph $G$ is a sequence $(v_0, e_1, v_1, \dots, v_{t-1}, \allowbreak e_{t}, v_{t})$ of alternating vertices $v_0,\dots,v_t$ and edges $e_1,\dots,e_t$ of $G$, such that $v_{i-1}$ and $v_i$ are the endpoints of $e_i$ for $1\leq i\leq t$. $v_0$ and $v_{t}$ are called the \emph{initial} and \emph{final} vertices respectively. A \emph{path} in $G$ is a walk whose vertices and edges both are distinct.
A \emph{cycle} in $G$ is a path whose initial and final vertices are identical. The \emph{size} of a path or a cycle is the number of edges in it.
{The definitions in this paragraph apply also to multigraphs.}

An \emph{edge-disjoint {multiset} of cycles} {in a multigraph} $G$ is a family of cycles $\mathcal{D} = \{d_1,\dots,d_k\}$ such that no edge belongs in more than one cycle.
In particular, a \textit{cycle decomposition} of a {multigraph} $G$ is an  edge-disjoint {multiset} of cycles $\mathcal{D}$ of $G$ such that $\underset{d \in \mathcal{D}}\cup E(d)=E(G)$.
Veblen's theorem~\cite[Theorem~2.7]{BondyGT} says that a {multigraph} admits a cycle decomposition if and only if it is even.
We say that a cycle decomposition is \textit{directed} if all of its cycles are directed.
For a plane graph $G$ and a cycle $c$ in $G$, denote \emph{the number of vertices in $V(G)$ enclosed by $c$} as $\chi(c)$.

\begin{definition}
\label{Def:DCD}
We say that the \emph{sign} of an edge-disjoint {multiset} of directed cycles $\mathcal{D}=\{d_1, \allowbreak \dots, d_k\}$ of an even plane {multigraph} $G$ is given by
\begin{equation}
\label{eqn:signDCD}
\sgn(\mathcal{D}):=\prod_{i=1}^k 
\begin{cases}
(-1)^{\chi(d_i)}
& \text{if $d_i$ has odd size and is directed clockwise,} \\
(-1)^{\chi(d_i)+1}
& \text{if either $d_i$ has even size, or has odd size} \\
& \hspace*{0.2cm} \text{and is directed anticlockwise.}
\end{cases}
\end{equation}
Note that this formula also defines the sign of a {directed cycle decomposition}. 
\end{definition}

\begin{example}
For the even plane graph $H$ shown in \cref{fig:graph}, the sign of its directed cycle decomposition $\{(1,2,3,4),(3,4,5),(5,6)\}$ shown in \cref{fig:CYCdec} is 
\begin{equation*}
    (-1)^{0+1}\times (-1)^{1}\times (-1)^{0+1}=-1.
\end{equation*}
\begin{figure}[h!]
        \centering
        \begin{subfigure}{.45\textwidth}
  \centering
  \includegraphics[width=0.8\linewidth]{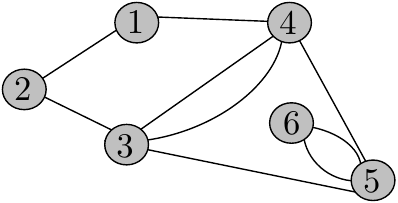}
\caption{A plane graph $H$ with parallel edges}     
  \label{fig:graph}
\end{subfigure}
\begin{subfigure}{.45\textwidth}
  \centering
  \includegraphics[width=0.8\linewidth]{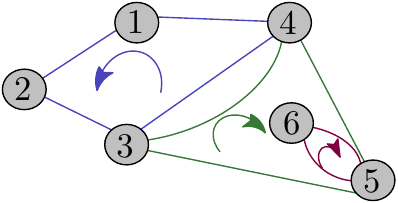}
  \caption{$H$ split as $(1234)(345)(56)$}
  \label{fig:CYCdec}
\end{subfigure}
        \caption{(A) A plane graph on 6 vertices and (B) a directed cycle decomposition of it.}
    \end{figure}

\end{example}
A \textit{trail} in a {multigraph} $G$ is a walk whose vertices can be repeated but edges are distinct. In particular, trails are allowed to self intersect at vertices.
A \textit{closed trail} is one whose initial and terminal vertices are the same. 
Therefore, a closed trail can be decomposed into a edge-disjoint {multiset} of cycles.
A \emph{directed closed trail} is a closed trail with a definite direction of traversal.

\begin{definition}
\label{def:cyc-dec-comp}
Let $T$ be a directed closed trail in a {multigraph} $G$, We say that a directed cycle decomposition $\mathcal{D}$ of $T$ is \emph{compatible} with $T$ if the direction of cycles in $\mathcal{D}$ is inherited from the direction of $T$.
\end{definition}

\begin{lemma}
\label{lem:IndCycDecomInDTrail}
Let $T$ be a closed directed trail in a plane {multigraph} $G$ with Pfaffian orientation $\mathcal{O}$. Then all cycle decompositions of $T$ compatible with it have the same sign.
\end{lemma}

\begin{proof}
The idea of this proof is similar to that of~\cite[Theorem~3.3]{Ayyer2015ASM}.
Let $\mathcal{D}=\{d_1,\dots,d_k\}$ be a directed cycle decomposition of $T$ compatible with it. Then the number of edges of $T$ oriented in the opposite direction to $T$ is given by
\begin{align}
\label{eqn:edRevDj}
\sum_{j=1}^k\bigg(\text{number of edges oriented in opposite direction of } d_j \bigg).
\end{align}
For $j\in[k]$, let $E_j$ and $F_j$ be the number of edges and faces enclosed by $d_j$ respectively. Since $\mathcal{O}$ is Pfaffian, the number of clockwise oriented edges on the boundary of any bounded face $f$ is odd (say $O_f$). Thus the number of clockwise oriented edges of $d_j$ is
\[
\sum_{\substack{f \text{ is a face in $G$} \\ \text{enclosed by $d_j$}}} O_f \, -E_j,
\]
because each edge enclosed by $d_j$ contributes to exactly two faces, one clockwise and one anticlockwise.  
Since $O_f$ is odd for any bounded face $f$, the above quantity has the same parity as
$F_j-E_j$.
Now, using the Euler characteristic on the plane graph enclosed by $d_j$, 
the number of clockwise oriented edges of $d_j$ and the number of vertices enclosed by $d_j$, which we called $\chi(d_j)$, have opposite parity.
 Thus the quantity in \eqref{eqn:edRevDj} is equal to $\sgn(\mathcal{D})$ given in \eqref{eqn:signDCD}.
\end{proof}

\begin{definition}
\label{def:ortd Cart}
The \emph{oriented Cartesian product} of naturally labeled oriented graphs $(G_1,\mathcal{O}_1), \allowbreak \dots, (G_k,\mathcal{O}_k)$ 
is the graph $G_1\square \cdots \square G_k$ with orientation $\mathcal{O}$ given 
as follows. For each $i \in [k]$, if $u_i\rightarrow u_i^\prime $ in $\mathcal{O}_i$, then 
$\mathcal{O}$ gives orientation
$(u_1,\dots,u_i,\dots,u_k) \rightarrow (u_1,\dots,u_i^\prime,\dots,u_k)$ if $u_{i+1}+u_{i+2}+\dots+u_k+(k-i)\equiv 0 \pmod 2$ and  $(u_1,\dots,u_i^\prime,\dots,u_k) \allowbreak \rightarrow (u_1,\dots,u_i,\dots,u_k)$ otherwise.
\end{definition}

If we assign the canonical orientation to the graph in \cref{fig:LabelledP4P3}, it can be thought of as an oriented Cartesian product of paths $P_4$ and $P_3$ which are labeled consecutively from one leaf to another.

\begin{definition}
\label{def:iproj}
The \emph{$i$-projection} of a subgraph $S$ of the
Cartesian product $G_1\square \cdots \square G_k$ is the {multigraph} obtained by contracting all but $G_i$-edges of $S$ and is denoted $\tilde{S}_{i}$.
\end{definition}

Let $G_1,\dots,G_k$ be $k$ plane simple naturally labeled graphs and $P$ 
be their Cartesian product.
Let $\ell=(w_0,w_1,\dots,w_{2s-1}, w_{2s}=w_0)$ be a directed even loop in $P$, and $\mathcal{D}_{i} $ be a cycle decomposition compatible with the $i$-projection $\tilde{\ell}_{i}$. 
For $i \in [k]$, let $\hat{G}^{(i)}$ be the graph $G_1 \square \cdots \square G_{i-1} \allowbreak \square G_{i+1} \square \cdots \square G_k$.
For $\hat{v} = (v_1,\dots,v_{i-1},v_{i+1},\dots,v_k) \in V(\hat{G}^{(i)})$, let $G_{i}(\hat{v})$ be the copy of $G_i$ in $P$ corresponding to $\hat{v}$
and let $e_i(\hat{v})$ be the number of edges lying both in $\ell$ and $G_{i}(\hat{v})$.
Now let 
\[
e_i = \sum_{\substack{\hat{v} \in V(\hat{G}^{(i)}) \\
v_{i+1}+\dots+v_k+(k-i)\equiv 1 \pmod 2} } e_i(\hat{v}).
\]
Then the \textit{sign} of $\ell$ is defined by
\begin{equation}
\label{eqn:NewIntforKcatesian}
    \sgn(\ell) := -\prod_{i=1}^{k-1}(-1)^{e_i} \,\,\prod_{j=1}^k \sgn (\mathcal{D}_{j}). 
\end{equation}
Note that the sign of $\ell$ is well-defined by \cref{lem:IndCycDecomInDTrail}. Now suppose that $P$ has been given vertex weights $x(w)$ for $w\in V(P)$ and edge weights $a_e$ for $e\in E(P)$.
Then the \textit{weight of the loop $\ell$} is defined as
\begin{equation}
\label{eqn:MMDMweight}
    w(\ell):=  \sgn(\ell) \prod_{e\in E(\ell)}a_e.
\end{equation}

Note that the orientation of a graph $G$ is not relevant for the definition of the loop-vertex configuration (defined in \cref{sec:MMDM}) on $G$. We will call a loop-vertex configuration an \emph{(extended) monopole-dimer configuration} when the underlying graph $G$ is a Cartesian product of simple plane graphs.

\begin{definition}
\label{Def:MDMonCarProd}
The \emph{(extended) monopole-dimer model} on the weighted Cartesian product $P = G_1\square \cdots \square G_k$
is the collection $\mathcal{L}$ of monopole-dimer configurations on $P$ where the weight of each configuration $C=(\ell_1,\dots,\ell_m; \,v_1,\dots,v_n)$ given by
\[
    w(C)=\prod_{i=1}^m w(\ell_i)\,\, \prod_{i=1}^n x(v_i).
\]
The \emph{(signed) partition function} of the monopole-dimer model on the Cartesian product $P$ is 
\[
    \mathcal{Z}_{P}:=\sum_{C\in \mathcal{L}}w(C).
\]
\end{definition}

From the above definition, it is clear that $\mathcal{Z}_{P}$ is independent of the orientations on $G_1,\dots,G_k$.
The following result is a generalisation of \cref{thm:pf-lvmodel} when $G$ is plane and $\mathcal{O}$ is Pfaffian.
Recall that $\mathcal{K}_P$ is the generalised adjacency matrix defined in \eqref{eqn:SAM} for $(P,\mathcal{O})$.

\begin{theorem} 
\label{thm:DetFor2} 
Let $G_1,\dots,G_k$ be $k$ simple plane naturally labeled graphs with Pfaffian orientations $\mathcal{O}_1,\dots,\mathcal{O}_k$ respectively. The (signed) partition function of the monopole-dimer model for the weighted oriented Cartesian product $(P,\mathcal{O})$ of $G_1,\dots,G_k$ is given by
\begin{equation}
    \mathcal{Z}_{P}=\det{\mathcal{K}_P}.
\end{equation}
\end{theorem}

The proof strategy is similar to that of \cref{thm:pf-lvmodel}.
\begin{proof}
Since $\mathcal{K}_P$ is the sum of a diagonal matrix and an antisymmetric matrix, the only terms contributing to $\det{\mathcal{K}_P}$ correspond to permutations which are product of even cycles and singletons, and hence are in bijective correspondence with monopole-dimer configurations on $P$. Thus, we only need to show that sign coming from an even cyclic permutation $(v_0,v_1,\dots,v_{2s-1}, \allowbreak v_{2s}=v_0)$ coincides with the sign of the corresponding directed loop
$\ell=(v_0,v_1,\dots,v_{2s-1},v_{2s}=v_0)$.
That is, we have to prove that
\[
(-1)^{\#\text{edges pointing in opposite direction of }\ell \text{ (under }\mathcal{O})+1}=\sgn\, \ell.
\]

Let $r$ be the number of edges pointing in opposite direction of $\ell$ under $\mathcal{O}$. Note that the contribution to $r$ comes from $k$ type of edges, $G_1$-edges, $G_2$-edges,$\dots, G_k$-edges in $\ell$. 
Since $\ell$ is a directed cycle, the $i$-th projection $\Tilde{\ell_{i}}$, of $\ell$ is a directed trail in $\Tilde{P_{i}}$ (which is just $G_i$ with multiple edges). Let $\mathcal{D}_{i}=\{d_{i,1},d_{i,2},\dots,d_{i,m_i}\}$ be a directed cycle decomposition compatible with $\Tilde{\ell_{i}}$ according to \cref{def:cyc-dec-comp}. Denote the number of edges in $d_{i,j}$, for $j\in[m_i]$, oriented under $\mathcal{O}$ in the direction opposite to it as $\varepsilon^{i,j}$.

Recall the notation $\chi(c)$ and $e_i$ from earlier in this section.
For $i \in [k-1]$, the edges contributing to $e_i$ have been reversed while defining $\mathcal{O}$ 
and thus the contribution of $G_i$-edges to $r$ is
\begin{multline*}
\sum_{j=1}^{m_i} \varepsilon^{i,j}
\equiv
    e_i \\
    +\sum_{j=1}^{m_i}
    \big( \text{number of edges in }d_{i,j} \text{ oriented under $\mathcal{O}_i$ in the direction opposite to it}\big) \pmod{2}.
\end{multline*}
By the proof of~\cref{lem:IndCycDecomInDTrail}, it follows that the number of clockwise oriented edges of $d_{i,j}$ under $\mathcal{O}_i$ and $\chi(d_{i,j})$ have opposite parity. Therefore,
\[
\sum_{j=1}^{m_i} \varepsilon^{i,j}
\equiv
    e_i+\sum_{j=1
    }^{m_i}\,\begin{cases}
     \chi(d_{i,j}) & 
    \text{if $d_{i,j}$ has odd size}\\ 
    & \text{and is directed clockwise,} \\
    \chi(d_{i,j})+1 & 
	\text{if either $d_{i,j}$ has even size, or has odd size}\\ 
	& \text{and is directed anticlockwise.} \\
    \end{cases}
    \pmod{2}.
\]
Now, by \cref{Def:DCD}, $G_i$-edges of $\ell$ contribute $(-1)^{e_i}\, \sgn\,\, \mathcal{D}_{i}$ to $(-1)^r$.
Similarly, the contribution of $G_k$-edges to $(-1)^r$ is $\sgn\,\, \mathcal{D}_{k}$ as we have not altered the directions coming from $\mathcal{O}_k$ in any copy of $G_k$. Thus,
\begin{align*}
    (-1)^{r}=\prod_{i=1}^{k-1}(-1)^{e_i}\sgn\,\, \mathcal{D}_{i}\times \sgn\,\, \mathcal{D}_{k} =\prod_{i=1}^{k-1}(-1)^{e_i} \,\,\prod_{s=1}^k \sgn\, \mathcal{D}_{s},
\end{align*}
resulting in $(-1)^{r+1}=  \sgn\, \ell$.
Hence, we get the (signed) partition function of the (extended) monopole-dimer model for the oriented cartesian product as a determinant.
\end{proof}

Recall that the the partition function of the monopole-dimer model is defined for (unoriented) Cartesian product of graphs in \cref{Def:MDMonCarProd}. The next result can thus be seen as an
analogue of the observation made at the end of \cref{sec:MMDM}.

\begin{corollary}
\label{cor:kProdisOrinInd}
Let $G_1,\dots,G_k$ be $k$ simple plane naturally labeled graphs with Pfaffian orientations $\mathcal{O}_1,\dots,\mathcal{O}_k$ respectively. Then the determinant  of the generalised adjacency matrix $\mathcal{K}_P$ of the oriented cartesian product, $(P,\mathcal{O})$, of $(G_1,\mathcal{O}_1),\dots,(G_k,\mathcal{O}_k)$ is independent of the Pfaffian orientations $\mathcal{O}_1,\dots,\mathcal{O}_k$.
\end{corollary}

\section{Monopole-dimer model on Cartesian products of plane bipartite graphs} 
\label{sec:MDMbipartite case}
Recall that a \textit{bipartite graph} is a graph $G$ whose vertex set can be partitioned into two subsets $X$ and $Y$ such that each edge of $G$ has one end in $X$ and other end in $Y$. Bipartite graphs only have cycles of even length. Since the direction of even cycles does not affect the sign in \cref{Def:DCD}, we can write the \emph{sign} of an edge-disjoint {multiset} of cycles $\mathcal{D}=\{d_1, \dots, d_k\}$ in an even bipartite plane {multigraph} $G$ as
$\sgn(\mathcal{D}):=\prod_{i=1}^k (-1)^{\chi(d_i)+1}$.

Note that the above formula also applies to a cycle decomposition of an even bipartite plane {multigraph}.
We will show that the sign of a cycle decomposition remains the same for all cycle decompositions of a plane bipartite even {multigraph}. For that we  first define some moves on two cycles in a decomposition.

Let $G$ be an even plane bipartite {multigraph} and $(c_1, \dots, c_{t})$ be an edge-disjoint {multiset} of cycles in $G$. 
Then we define the following moves transforming one {multiset} of cycles into another.
For each move, we also calculate the change in the sign of this {multiset} of cycles.
\begin{enumerate}

    \begin{figure}[h!]
        \centering
        \includegraphics[scale=0.8]{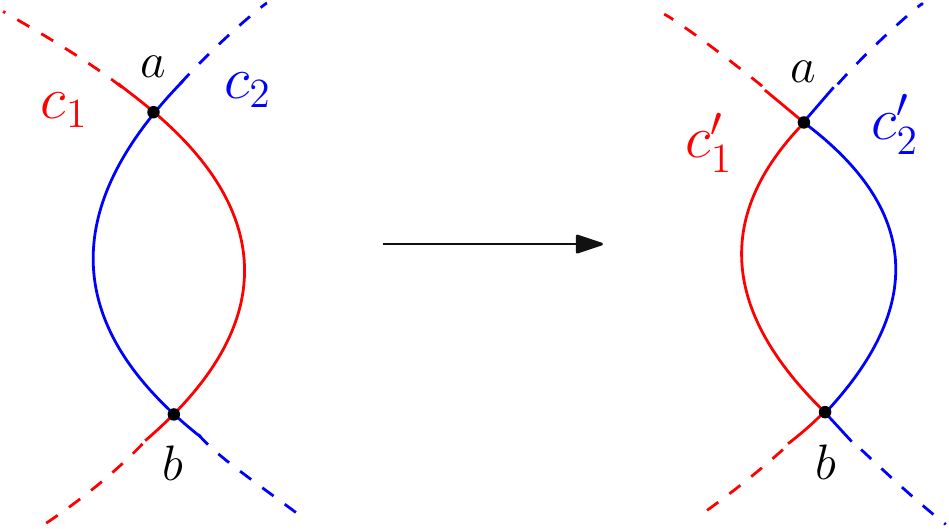}
        \caption{The $M_1$-move from \cref{it:M0}. Here, the dotted blue and red lines indicate that they are allowed to intersect each other.}
        \label{fig:M1 move}
    \end{figure}

    \item 
    \label{it:M0}
    The \emph{$M_1$-move} changes cycles $c_1$, $c_2$ into cycles $c_1^\prime$, $c_2^\prime$ as shown in \cref{fig:M1 move}.
Let $v$ (resp. $v'$) be the number of internal vertices lying on the blue (resp. red) solid path from $a$ to $b$ along $c_2$ (resp. $c_1$) in the left side of \cref{fig:M1 move}.
Let $u$ be the number of vertices enclosed by the cycle formed by these two paths. Then
\begin{align*}
\chi(c_1^\prime)+\chi(c_2^\prime) &=\chi(c_1)-u-v+\chi(c_2)-u-v^\prime\\
   &\equiv \chi(c_1)+\chi(c_2) \pmod 2\\
   &\equiv(\chi(c_1)+1)+(\chi(c_2)+1) \pmod 2,
    \end{align*}
where we have used the fact that $v+v^\prime\equiv 0$ (as $G$ is bipartite) in the second line.
Thus, performing the $M_1$-move in this {multiset} of cycles preserves its sign.

\begin{figure}[h!]
\centering
    \includegraphics[scale=0.8]{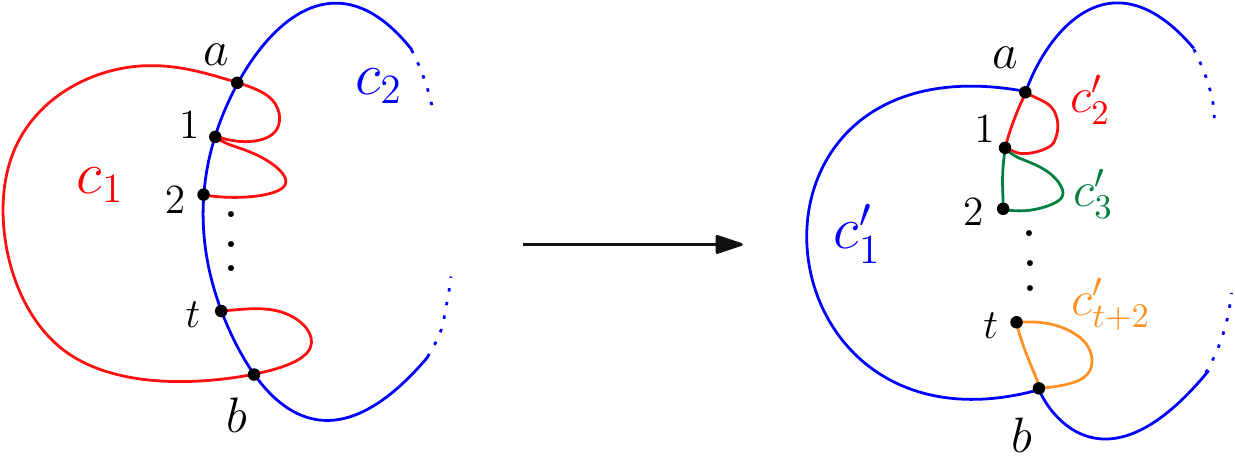}
\caption{The $M_2$-move from \cref{it:Mt}. The dotted arc of $c_2$ on the right indicates that it can intersect with the part of $c_1$ from $a$ to $b$ on the right side.}
\label{fig:M2 move}
\end{figure}

    \item 
    \label{it:Mt}
   Let $c_1,c_2$ intersect as shown on the left side of \cref{fig:M2 move}. Without loss of generality, the left arc of $c_1$ strictly between $a$ and $b$ does not intersect $c_2$, and the right arc of $c_1$ strictly between $a$ and $b$ intersects the left arc of $c_2$ strictly between $a$ and $b$ only at the $t$ points shown.
    The \emph{$M_2$-move} changes cycles $c_1,c_2$ into $(t+2)$ cycles $c_1^\prime,c_2^\prime,\dots,c_{t+2}^\prime$ as shown in the right side of \cref{fig:M2 move}.
Then, by considering internal vertices in all regions, the sign of the latter {multiset} is given by
\begin{align*}
\sum_{i=1}^{t+2}(\chi(c_i^\prime)+1) &=\chi(c_1^\prime)+\sum_{i=2}^{t+2}\chi(c_i^\prime)+(t+2)\\
&=\bigg(\chi(c_1)+\chi(c_2)+t-\sum_{i=2}^{t+2}\chi(c_i^\prime)\bigg)+\sum_{i=2}^{t+2}\chi(c_i^\prime)+t+2 \\
      &\equiv(\chi(c_1)+1)+(\chi(c_2)+1) \pmod 2.
    \end{align*}
Thus, the $M_2$-move also preserves the sign of the {multiset} of cycles.

    \begin{figure}[h!]
        \centering
        \includegraphics[scale=0.8]{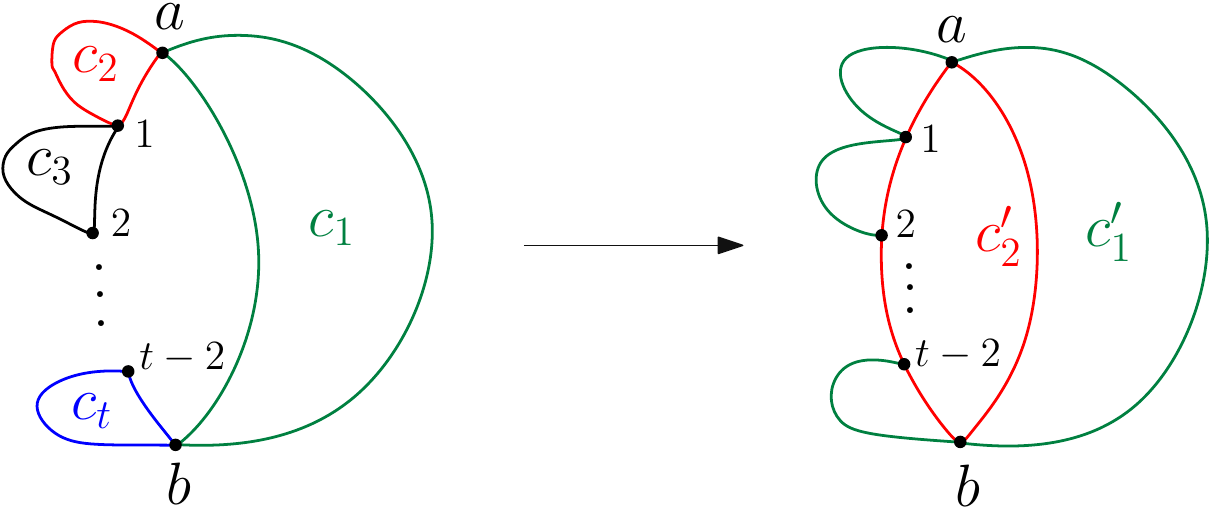}
        \caption{The $M_3$-move from \cref{it:Mt'}.}
        \label{fig:M3 move}
    \end{figure}

    \item 
    \label{it:Mt'}
    Fot $t \geq 2$, suppose $c_1,\dots,c_{t}$ are cycles in the {multigraph} $G$ such that they form a closed chain as shown on the left side in~\cref{fig:M3 move}. Note that these cycles do not intersect at points other than those shown in the figure. 
    The \emph{$M_3$-move} converts these $t$ cycles into two cycles namely $c_1^\prime,c_2^\prime$ as shown in the right side of~\cref{fig:M3 move}.
 Then the sign of the latter {multiset} of cycles is
     \begin{align*}
\sum_{i=1}^{2} \left( \chi(c_i^\prime)+1 \right) &\equiv \chi(c_1^\prime)-\chi(c_2^\prime) \pmod 2,\\
  &\equiv \sum_{i=1}^{t}\chi(c_i)+(\text{size of }c_2^\prime)-t \pmod 2 ,
  \end{align*}
because $c_1'$ encloses all the vertices except  $a,b,1,\dots,(t-2)$ lying on $c_2'$. Since $G$ is bipartite, it can only have cycles of even size. Thus
      \begin{align*}
\sum_{i=1}^{2} \left( \chi(c_i^\prime)+1 \right)  &\equiv \sum_{i=1}^{t}\chi(c_i)+t \pmod 2 \\
  &\equiv \sum_{i=1}^{t}(\chi(c_i)+1) \pmod 2 .
    \end{align*}   
Again, the sign of the {multiset} is preserved under the $M_3$-move.

\end{enumerate}

We have thus shown that performing any sequence of moves of the form \cref{it:M0,it:Mt,it:Mt'} will not affect the sign of a {multiset} of cycles.

Recall that a \emph{bridge} or \emph{cut edge} in a {multigraph} $G$ is an edge whose deletion increases the number of connected components in $G$.
Let $G$ be a connected even plane {multigraph}. Then $G$ cannot have a bridge~\cite[Exercise 3.2.3]{BondyGT} and hence the boundary of the outer face (being a closed trail) can be decomposed into cycles $c_1,\dots,c_k$ such that $|V(c_{i})\cap V(c_{j})| \leq 1$ for all $i,j \in [k]$.

\begin{definition}
 Let $G$ be a connected even plane {multigraph} and $C$ be the boundary of the outer face consisting of cycles $c_1,\dots,c_k$. Then an \emph{outer cycle decomposition} of $G$ is a cycle decomposition of $G$ containing $c_1,\dots,c_k$ and the latter will be called \emph{boundary cycles}.
\end{definition}

The next result is a crucial step towards the main result of this section.

\begin{lemma}
 \label{lem:BdryCyclDecLemma}
Let $G$ be a connected bipartite even plane {multigraph}. Then for any cycle decomposition $\mathcal{D}$ of $G$, there exists an outer cycle decomposition $\mathcal{D}^\prime$ of $G$ with same sign. 
 \end{lemma}

 \begin{proof}
 Suppose $c_1,\dots,c_k$ are the boundary cycles.
 If $k> 1$, we can work separately with each subgraph of $G$ lying inside the cycle $c_j$ for each $j\in[k]$. Thus, without loss of generality, we can assume that there is just a single cycle $c$ (say). If $\mathcal{D}$ contains $c$, there is nothing to prove. So assume $\mathcal{D}$ does not contain $c$. Then there exist at least two cycles
 in $\mathcal{D}$ which will intersect $c$ in some edge(s). 
 There are two possibilities now depending on whether the cycles above intersect each other more than once or not.
 
 \begin{enumerate}

     \item  If there are two cycles among these, say $\ell_1$ and $\ell_2$, which intersect each other in more than one point, say $a$ and $b$, then $\ell_1,\ell_2$ will look as in \cref{fig:OuterCycDecPart1}. 
     Let the bottom arc of $\ell_1$ joining $x$ to $y$ intersect $\ell_2$ in $t$ additional points as shown. 
     First perform the necessary number of $M_1$-moves to reach the stage in \cref{fig:OuterCycDecPart2}.
     At this point, the top part of the $\ell'_2$ cycle between $a$ and $b$ lies on the same side of the bottom part of the $\ell'_1$ cycle.
     Now, perform an $M_2$-move to increase the number of edges of $c$ covered by $\ell'_1$ as depicted in \cref{fig:OuterCycDecPart3}. 
     Since these moves preserve the sign, the cycle decomposition containing $(t+2)$ transformed cycles in place of two original cycles will have the same sign.

\begin{figure}[h!]
\begin{subfigure}{.45\linewidth}
\centering
\includegraphics[scale=0.5]{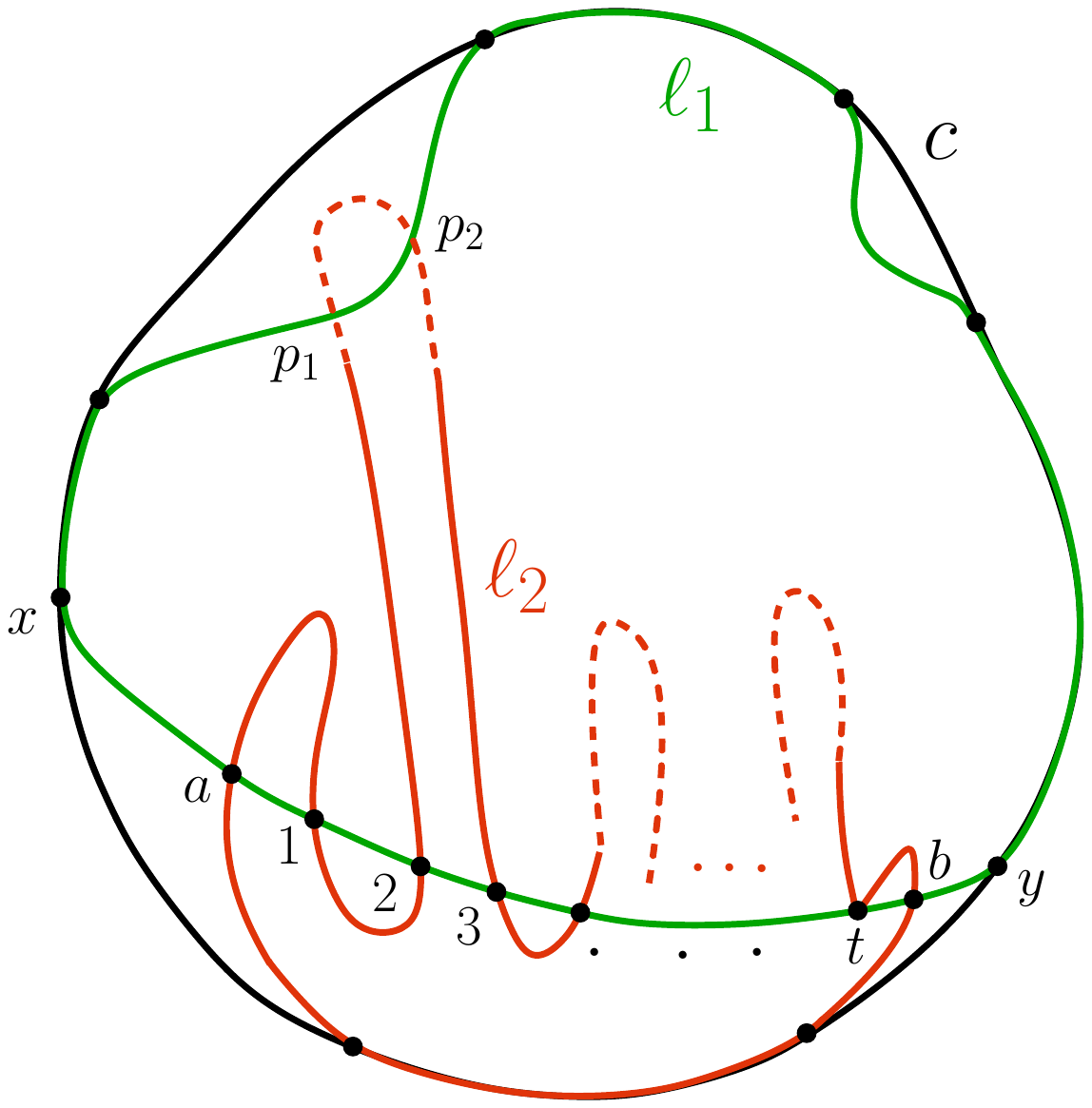}
\caption{The original cycles $\ell_1$ and $\ell_2$.}
 \label{fig:OuterCycDecPart1}
\end{subfigure}
\begin{subfigure}{.45\linewidth}
\centering
\includegraphics[scale=0.5]{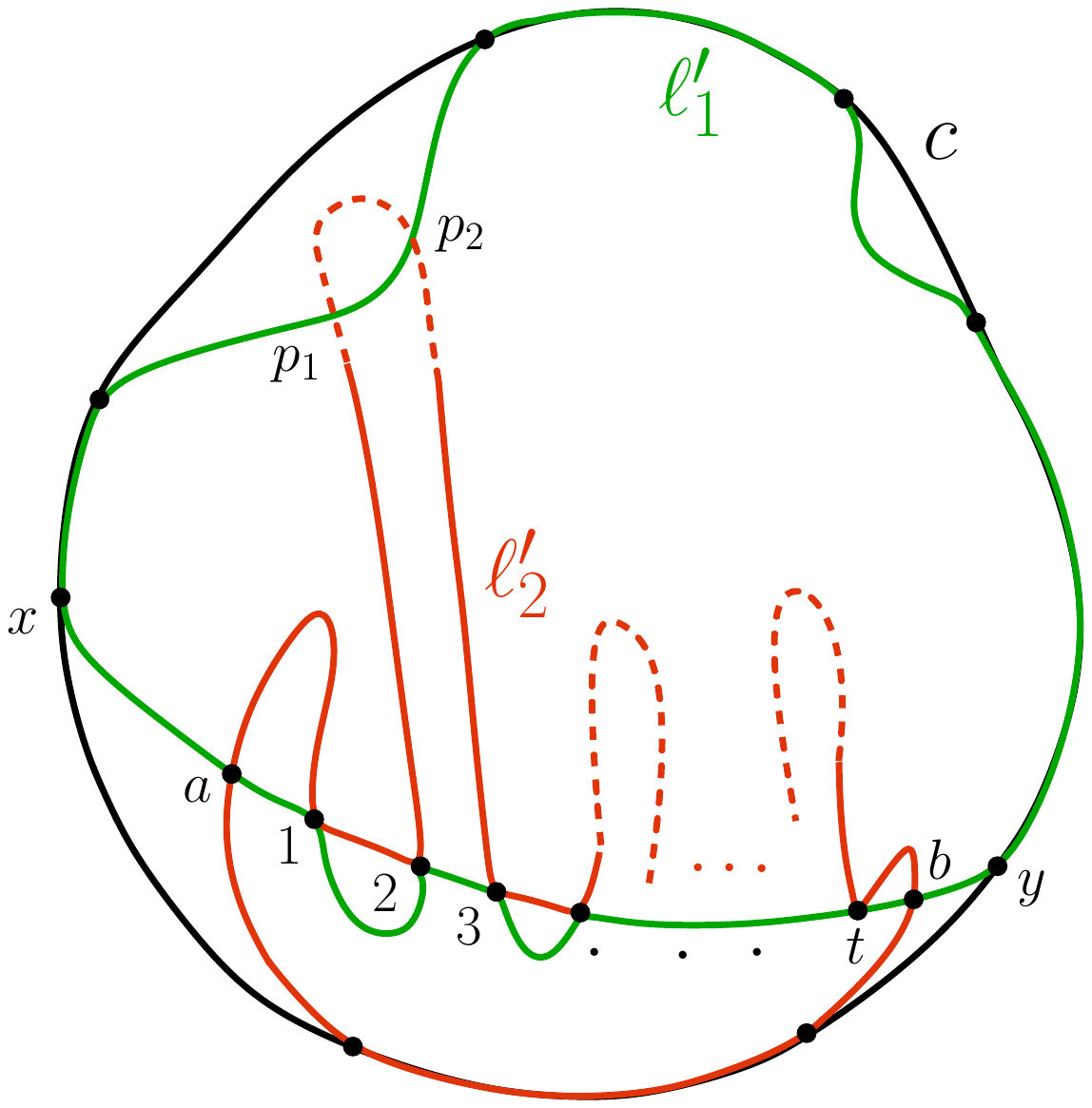}
  \caption{The new cycles $\ell'_1$ and $\ell'_2$ after several $M_1$-moves.}
   \label{fig:OuterCycDecPart2}
\end{subfigure}\\[1ex]
\begin{subfigure}{\linewidth}
\centering
\includegraphics[scale=0.5]{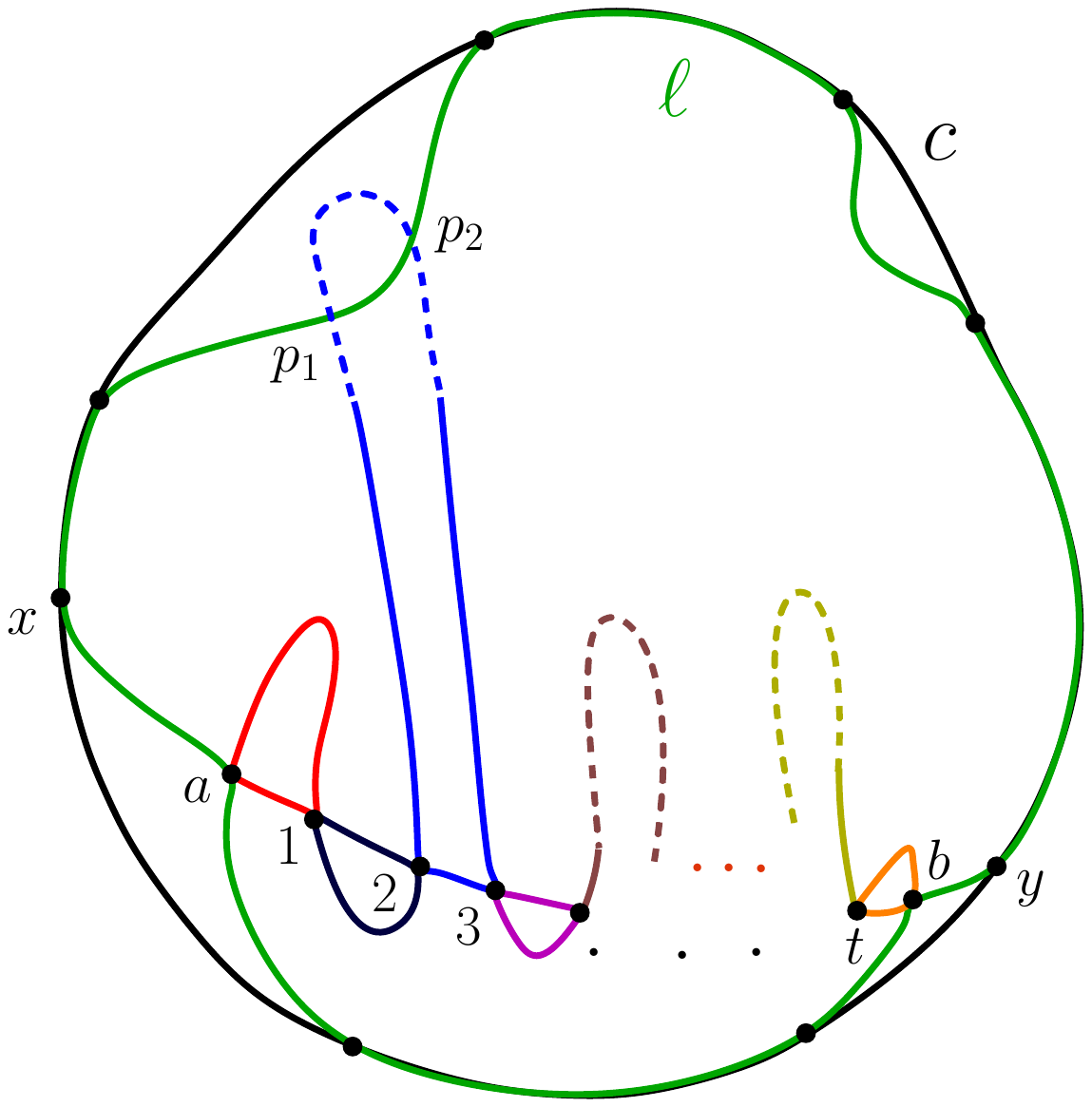}
  \caption{The final cycles after an $M_2$-move. Note that there are a total of $t+2$ cycles now.}
  \label{fig:OuterCycDecPart3}
\end{subfigure}
\caption{Two cycles $\ell_1$ and $\ell_2$ intersecting the boundary cycle $c$, and intersecting each other in more than one point in \cref{fig:OuterCycDecPart1}. In \cref{fig:OuterCycDecPart3}, the cycle $\ell$ intersects $c$ in more edges than $\ell_1$. }
\end{figure}

     \item If no two cycles of $\mathcal{D}$ which have a common edge with $c$ intersect each other in more than one point, then we have a certain number, say $t \geq 3$, of cycles intersecting $c$. 
     Focus on one of the cycles, $\ell_1$ say. It will intersect
     another cycle at a vertex of $c$. Call it $\ell_2$.
     Now following $\ell_1$ in the interior, find the first of these $t$ cycles and call it $\ell_t$.
     The situation will look as in left of  \cref{fig:M3 move used in lemma}.
     Now perform an $M_3$-move to increase the number of edges of $c$ covered by $\ell_1$ and arrive at the right of \cref{fig:M3 move used in lemma}. The resulting cycle decomposition will have the same sign.

\begin{figure}[h!]
        \centering
        \begin{subfigure}{.45\textwidth}
  \centering
  \includegraphics[scale=0.5]{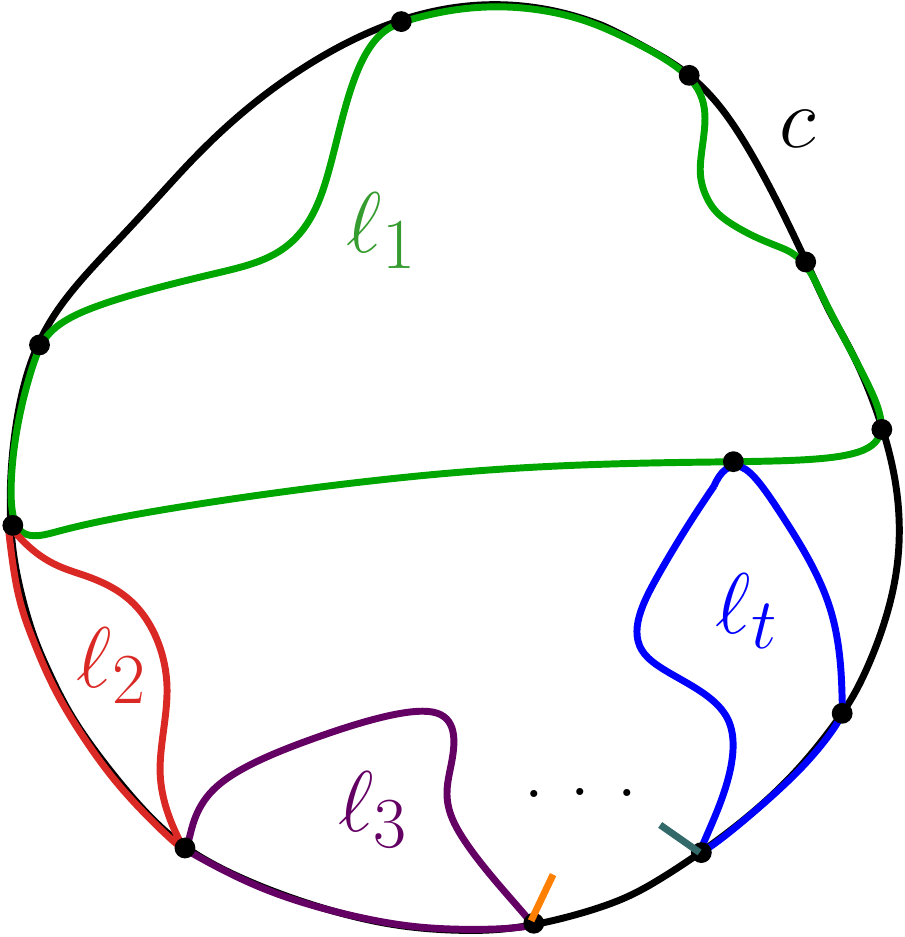}
\end{subfigure}
\begin{subfigure}{.45\textwidth}
  \centering
  \includegraphics[scale=0.5]{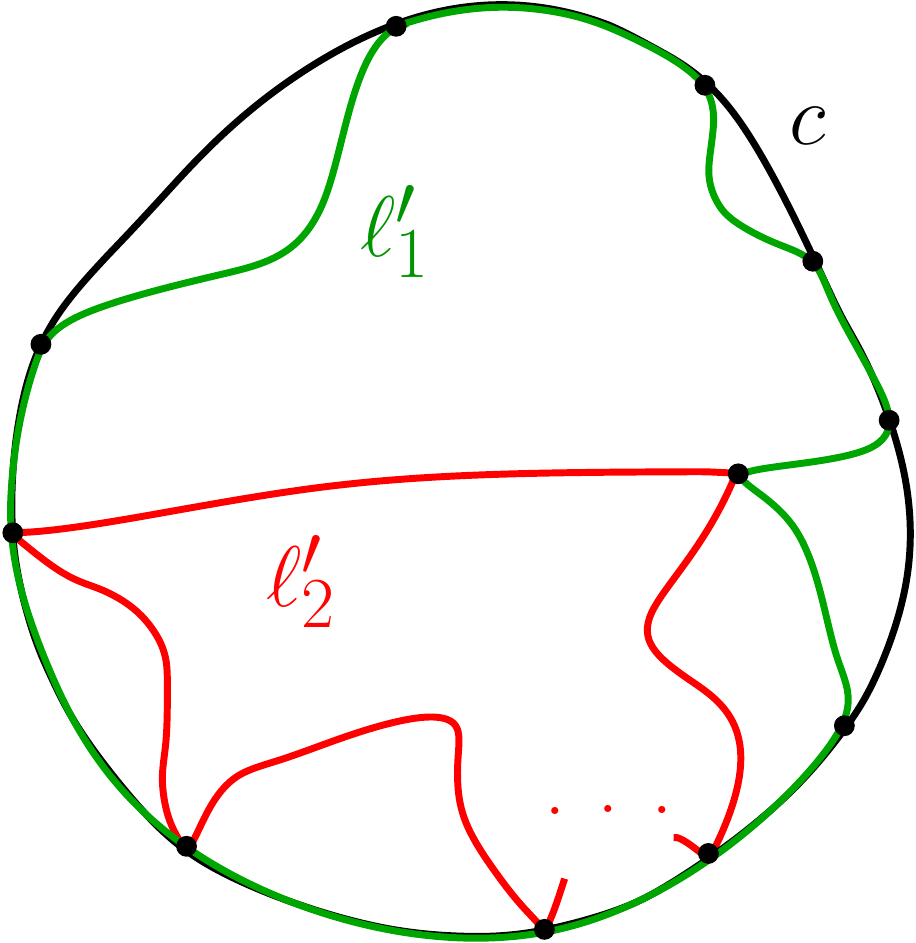}
\end{subfigure}
        \caption{A cycle decomposition where the cycles intersecting the boundary cycle are shown before and after an $M_3$-move.}
        \label{fig:M3 move used in lemma}
    \end{figure} 
 \end{enumerate}

Apply these cases inductively. 
Notice that we might need to alternate between these two. In each case, the number of edges in the intersection of $\ell_1$ and $c$ increases. As the number of edges in $c$ is finite, the process of performing these moves will eventually stop and we will end up with a cycle decomposition containing $c$ with the sign same as that of $\mathcal{D}$.
 \end{proof}
 
Now we will see a result analogous to~\cref{lem:IndCycDecomInDTrail} in the case of bipartite graphs.

\begin{theorem}
\label{thm:IndCycDecom}
Let $G$ be a connected bipartite even plane {multigraph}. Then all cycle decompositions $\mathcal{D}$ of $G$ will have same sign.
\end{theorem}

\begin{proof}
Since $G$ is even and connected, the boundary of the outer face of $G$ can be decomposed into cycles. 
For simplicity, we suppose the boundary is a single cycle $c$. If not, the argument below extends in an obvious way to each component of the boundary.

Let $\mathcal{D}$ be a cycle decomposition of $G$. Using \cref{lem:BdryCyclDecLemma}, we obtain another cycle decomposition $\mathcal{D}_1$ containing $c$ which has the same sign as $\mathcal{D}$. Let $G_1$ be obtained from $G$ by removing all the edges of $c$ and the resulting isolated vertices.
Note that although $G_1$ can be disconnected, the regions enclosed by its connected components $G_{1,1}, G_{1,2}, \dots, G_{1,t}$ will not intersect.
Now, $\mathcal{D}_1\setminus\{c\}$ is a cycle decomposition of $G_1$. Again using \cref{lem:BdryCyclDecLemma} on $G_1$,
 we obtain a cycle decomposition $\mathcal{D}_2$ of $G$ containing $c$ and $d_{1,1}, d_{1,2}, \dots , d_{1,t}$, the boundary cycles of $G_{1,1}, G_{1,2}, \dots , G_{1,t}$ respectively, such that $\sgn \mathcal{D}_1 = \sgn \mathcal{D}_2$. Now remove $d_{1,1}, d_{1,2}, \dots , d_{1,t}$ from $G_1$ to obtain $G_2$ and
continue this process. Since $G$ is finite, this process must stop. In fact, it will stop at the cycle decomposition obtained by successively including outer boundaries of $G_1, G_2 $ and so on.
\end{proof}

Recall the $i$-projection, the sign of a directed loop and the notation $e_i$ defined in \cref{def:iproj}, \eqref{eqn:NewIntforKcatesian} and~\cref{sec:MainRe} respectively.
By the fact that bipartite graphs only have even cycles and by \cref{thm:IndCycDecom}, we have the following result.

\begin{corollary}
Let $G_1,\dots,G_k$ be plane simple naturally labeled bipartite graphs and $P$ be their Cartesian product.
Let $\ell=(w_0,w_1,\dots,w_{2s-1}, w_{2s}=w_0)$ be a directed even loop in $P$ and $\mathcal{D}_{i} $ be an arbitrary cycle decomposition of the $i$-projection $\tilde{\ell}_{i}$ for $i\in[k]$. 
Then
\[
\sgn(\ell) = -\prod_{i=1}^{k-1}(-1)^{e_i} \,\,\prod_{j=1}^k \sgn (\mathcal{D}_{j}),
\]
and is well-defined. 
In particular, there is no restriction on the choice of cycle decomposition of any $i$-projection in 
the monopole-dimer model for Cartesian product of bipartite graphs.
\end{corollary}

\section{Three-dimensional grids} 
\label{sec:MDM3D}
Recall that $P_n$ is the path graph on $n$ vertices. Assign to $P_n$ the natural labelling increasing from one leaf to another, and denote its oriented variant with the canonical orientation as $(P_n,\mathcal{O}_n)$.
Consider the two-dimensional grid graph $P_{\ell}\square P_m=\{(p,q) \mid p \in [\ell],q \in [m]\}$ whose vertex $(p,q)$ has label 
$2s\ell+p$ if $q=2s+1$ and $2s\ell-p+1$ if $q=2s$.
Such a `snake-like' labelling is known as a \emph{boustrophedon labelling}~\cite{hammersley-menon-1970}.
\cref{fig:LabelledP4P3} shows this labelling on $P_4\square P_3$. 
With the canonical orientation, denote this graph as $(P_{\ell}\square P_m,\mathcal{O}_{\ell,m})$. 
For the purposes of our next result, we will think of the Cartesian product of $P_{\ell} \square P_m$ with $P_n$ as embedded in $\mathbb{Z}^3$ where the coordinate axes $x, y, z$ are aligned parallel to the edges in $P_\ell, P_m$ and $P_n$ respectively.
\begin{theorem} 
\label{thm:PF3Dgrid}
Let $(G,\mathcal{O})$ be the oriented Cartesian product of $(P_{\ell}\square P_m,\mathcal{O}_{\ell,m})$ with $(P_n,\mathcal{O}_n)$. Let vertex weights be $x$ for all vertices of $G$, and edge weights be $a$, $b$, $c$ for the edges 
along the three coordinate axes. Then the partition function of the monopole-dimer model on $G$ is given by
        \begin{multline*}
     \mathcal{Z}_{G}\equiv \mathcal{Z}_{\ell,m,n}=\prod_{j=1}^{\lfloor n/2 \rfloor} \prod_{s=1}^{\lfloor m/2 \rfloor}\prod_{k=1}^{\lfloor \ell/2 \rfloor}\bigg(x^2+4a^2\cos^2{\frac{\pi k}{\ell+1}}+4b^2\cos^2{\frac{\pi s}{m+1}}+4c^2\cos^2{\frac{\pi j}{n+1}}\bigg )^4 \\ 
    \times \begin{cases}
        1 & \ell,n,m\in 2\mathbb{N},\\
       T_{n,m}^2(b,c;x) &\ell\notin 2\mathbb{N}, m,n\in 2\mathbb{N},\\
       T_{n,\ell}^2(a,c;x) &\ell,n\in 2\mathbb{N} , m\notin 2\mathbb{N},\\
       T_{n,m}^2(b,c;x)\, T_{n,\ell}^2(a,c;x)\, S_n(c;x) & \ell,m\notin 2\mathbb{N} , n\in 2\mathbb{N},\\
       T_{m,\ell}^2(a,b;x) &\ell,m\in 2\mathbb{N}, n\notin 2\mathbb{N},\\
       T_{n,m}^2(b,c;x)\, T_{m,\ell}^2(a,b;x) \, S_m(b;x)  &\ell,n\notin 2\mathbb{N} , m\in 2\mathbb{N},\\
       T_{n,\ell}^2(a,c;x)\, T_{m,\ell}^2(a,b;x) \, S_{\ell}(a;x) &\ell\in 2\mathbb{N}, m,n\notin 2\mathbb{N},\\
       x \ T_{n,m}^2(b,c;x)\, T_{n,\ell}^2(a,c;x)\, T_{m,\ell}^2(a,b;x)\, S_n(c;x) \, S_m(b;x)\, S_{\ell}(a;x) 
       & \ell,m,n\notin 2\mathbb{N},
    \end{cases}
   \end{multline*}
   where
   \begin{equation*}
    S_n(c;x)=\prod_{k=1}^{\lfloor n/2\rfloor}\bigg (x^2+4c^2\cos^2{\frac{\pi k}{n+1}}\bigg),
    \end{equation*}
and
\begin{equation*}
    T_{n,\ell}(a,b;x)=\prod_{j=1}^{\lfloor n/2 \rfloor}\prod_{k=1}^{\lfloor \ell/2 \rfloor} \bigg(x^2+4a^2\cos^2{\frac{\pi k}{\ell+1}}+4b^2\cos^2{\frac{\pi j}{n+1}}\bigg).
\end{equation*}
\end{theorem}

\begin{figure}[h!]
    \centering
    \includegraphics[scale=0.80]{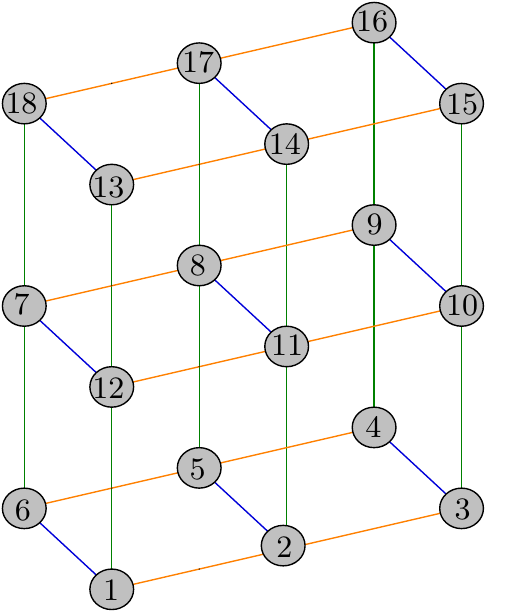}
    \caption{The boustrophedon labelling on $P_3\square P_2 \square P_3$.}
    \label{fig:labelEg}
\end{figure}

\begin{remark}
\label{Rem:boustro_label}
The \emph{boustrophedon labelling} that induces the orientation $\mathcal{O}$ over the graph $G$ is as follows. The vertex $(p,q,r)$ has label
\begin{equation*}
\begin{cases}
2t\ell m+{2s\ell}+p &  q=2s+1, r=2t+1,\\
      2t\ell m+2s\ell-p+1 &  q=2s, r=2t+1,\\
      2t\ell m-{2s\ell}-p+1 &  q=2s+1, r=2t,\\
      2t\ell m-2s\ell+p &  q=2s, r=2t,
\end{cases}
\end{equation*}
where $p \in [\ell],q \in [m]$ and $r \in [n]$. 
\cref{fig:labelEg} shows this labelling on the graph $P_3\square P_2 \square P_3$. 
\end{remark}

\begin{proof}
The signed adjacency matrix for the graph $(G,\mathcal{O})$ with the above labelling can be written as
\begin{equation}
    \mathcal{K}_{\ell,m,n}= I_n \otimes I_m \otimes T_{\ell}(-a,x,a)+I_n \otimes T_m(-b,0,b) \otimes J_{\ell}+T_n(-c,0,c) \otimes J_m \otimes J_{\ell},
\end{equation}
where $T_k(-s,z,s)$ is the $k\times k$ tridiagonal Toeplitz matrix with diagonal entries $z$, subdiagonal entries $-s$ and superdiagonal entries $s$. Let $J_k$ be the $k\times k$ antidiagonal matrix with all antidiagonal entries equal to $1$ and let $\iota = \sqrt{-1}$.
Transform $T_k(-s,z,s)$ into the diagonal matrix 
$$
D_k =\mathrm{diag}\bigg(z+2\iota s\cos{\frac{\pi}{k+1}},
\dots,z+2\iota s\cos{\frac{k\pi}{k+1}}\bigg)
$$ 
using the standard unitary similarity transformation~\cite[Section 4]{Fisher} given by the matrix $u_k$ whose entries are given by
$(u_k)_{p,q}=\sqrt{\frac{2}{k+1}}\iota^{p}\sin (\pi pq/(k+1))$. A short calculation shows that
\begin{equation}
\label{Eqn:UTforToeplitz}
(u_k)^{-1}J_k u_k=
\iota^{k-1}\begin{pmatrix}
   & & &(-1)^{k-1}\\
     & & (-1)^{k-2}&\\
     &\iddots & &\\
    (-1)^0& & & 
  \end{pmatrix}.
\end{equation}
Let $\sim$ denote the equivalence relation on matrices given by similarity. Then, using the unitary transform $u_n\otimes u_m \otimes u_{\ell}$, we find that 
\begin{multline*}
     \mathcal{K}_{\ell,m,n} \sim I_n \otimes I_m \otimes\, {\mathrm{diag}\bigg( x+2\iota a\cos{\frac{\pi}{\ell+1}},\dots,x+2\iota a\cos{\frac{\ell\pi}{\ell+1}} \bigg)}
  \\
    +I_n \otimes 2\iota b\,
{\mathrm{diag}\bigg( \cos{\frac{\pi}{m+1}},\dots,\cos{\frac{m\pi}{m+1}}\bigg)}
   \otimes \iota^{\ell-1}
\begin{pmatrix}
   & &(-1)^{\ell-1}\\
     &\iddots &\\
    (-1)^0& & 
  \end{pmatrix}\\
   + 2\iota c
   \,
{\mathrm{diag}\bigg( \cos{\frac{\pi}{n+1}},\dots,\cos{\frac{n\pi}{n+1}}\bigg)} \\
\otimes \iota^{m-1}\begin{pmatrix}
   & &(-1)^{m-1}\\
     &\iddots &\\
    (-1)^0& & 
  \end{pmatrix} 
  \otimes \iota^{\ell-1}\begin{pmatrix}
   & &(-1)^{\ell-1}\\
     &\iddots &\\
    (-1)^0& & 
  \end{pmatrix}.
\end{multline*}
Note that each term in the above sum is a block diagonal matrix as the first matrix in each tensor factor is diagonal. 
Therefore, $\mathcal{K}_{\ell,m,n}$ is similar to an $n\times n$ block diagonal matrix with the block 
$F_j$ for $j\in [n]$ given by
\begin{multline*}
   F_j = I_m \otimes\, \mathrm{diag}\bigg( x+2\iota a\cos{\frac{\pi}{\ell+1}},\dots,x+2\iota a\cos{\frac{\ell\pi}{\ell+1}} \bigg) \\
   + \begin{pmatrix}
   2\iota b\cos{\frac{\pi}{m+1}}& & &(-1)^{m-1}2\iota^{m}c\cos{\frac{j\pi}{n+1}}\\
    &\ddots&\iddots&\\
   &\iddots&\ddots&\\
   (-1)^02\iota^{m}c\cos{\frac{j\pi}{n+1}}&&& 2\iota b\cos{\frac{m\pi}{m+1}}
  \end{pmatrix} \otimes \iota^{\ell-1}\begin{pmatrix}
   & &(-1)^{\ell-1}\\
     &\iddots &\\
    (-1)^0& & 
  \end{pmatrix}. 
\end{multline*}
Define
\[
 \lambda^{m,n}_{s,j}=\sqrt{4b^2 \cos^2{\frac{s\pi}{m+1}}+4c^2 \cos^2{\frac{j\pi}{n+1}}},
 \quad s \in [m],\, j \in [n], 
 \]
and let 
\[
D_j=\begin{cases}
      \mathrm{diag}\Big(\iota\lambda^{m,n}_{1,j},-\iota \lambda^{m,n}_{1,j},\iota\lambda^{m,n}_{2,j},-\iota \lambda^{m,n}_{2,j},\dots,\iota\lambda^{m,n}_{\frac{m}{2},j},-\iota \lambda^{m,n}_{\frac{m}{2},j}\Big) &  m\text{ even},\\
      \mathrm{diag}\Big(\iota\lambda^{m,n}_{1,j},-\iota \lambda^{m,n}_{1,j},\iota\lambda^{m,n}_{2,j},-\iota \lambda^{m,n}_{2,j},\dots,\iota\lambda^{m,n}_{\lfloor\frac{ m}{2}\rfloor,j},-\iota \lambda^{m,n}_{\lfloor\frac{ m}{2}\rfloor,j},\iota \lambda^{m,n}_{\frac{ m+1}{2},j}\Big) &  m \text{ odd}.
  \end{cases}
  \]
The matrix 
\[
\begin{pmatrix}
   2\iota b\cos{\frac{\pi}{m+1}}& & &(-1)^{m-1}2\iota^{m}c\cos{\frac{j\pi}{n+1}}\\
    &\ddots&\iddots&\\
   &\iddots&\ddots&\\
   (-1)^02\iota^{m}c\cos{\frac{j\pi}{n+1}}&&& 2\iota b\cos{\frac{m\pi}{m+1}}
  \end{pmatrix},
  \]
when diagonalized, becomes equal to $D_j$. Thus, $F_j$ becomes similar to
  \begin{equation}
  \label{eqn:BlockofKG}
     I_m \otimes\, \mathrm{diag}\bigg( x+2\iota a\cos{\frac{\pi}{\ell+1}},\dots,x+2\iota a\cos{\frac{\ell\pi}{\ell+1}} \bigg)+
D_j  \otimes \iota^{\ell-1}\begin{pmatrix}
   & &(-1)^{\ell-1}\\
     &\iddots &\\
    (-1)^0& & 
  \end{pmatrix}.
  \end{equation}
Again, as the first matrix in both tensor products of \eqref{eqn:BlockofKG} is diagonal, each block $F_j$ is similar to a block diagonal matrix given by
$$ 
F_j\sim \begin{cases}
    \mathrm{diag}\big(F^+_{1,j},F^-_{1,j},\dots, F^+_{\frac{m}{2},j},F^-_{\frac{m}{2},j}\big) & m\text{ even,}\\
    \mathrm{diag}\Big(F^+_{1,j},F^-_{1,j},\dots,F^+_{\frac{m-1}{2},j},F^-_{\frac{m-1}{2},j},F^+_{\frac{m+1}{2},j}\Big) & m\text{ odd,}
\end{cases} 
$$
where
\begin{align*}
  F^{\pm}_{s,j} &=\mathrm{diag} \left( x+2\iota a\cos{\frac{\pi}{\ell+1}},\dots,x+2\iota a\cos{\frac{\ell\pi}{\ell+1}} \right)
  \pm \iota^{\ell} \lambda^{m,n}_{s,j}\begin{pmatrix}
   & &(-1)^{\ell-1}\\
     &\iddots &\\
    (-1)^0& & 
  \end{pmatrix}\\
 &=\begin{pmatrix}
   x+2\iota a\cos{\frac{\pi}{\ell+1}}& & &\pm (-1)^{\ell-1}\iota^{\ell}\lambda^{m,n}_{s,j}\\
    &\ddots&\iddots&\\
   &\iddots&\ddots&\\
   \pm(-1)^{0}\iota^{\ell}\lambda^{m,n}_{s,j}&&& x+2\iota a\cos{\frac{\ell\pi}{\ell+1}}
  \end{pmatrix}.
\end{align*}
Now defining
\[
Y^{\pm}_{k,s,j}=\begin{pmatrix}
  x+2\iota a\cos{\frac{k\pi}{\ell+1}}&\pm(-1)^{\ell-k}\iota^l\lambda^{m,n}_{s,j}\\
\pm(-1)^{k-1}\iota^{\ell}\lambda^{m,n}_{s,j} & x-2\iota a\cos{\frac{k\pi}{\ell+1}}
  \end{pmatrix}, \quad 1 \leq s\leq \lfloor (m+1)/2\rfloor,
\]
and performing simultaneous row and column interchanges on $F^{\pm}_{s,j}$,
we can write
\begin{align*}
  F^{\pm}_{s,j} &\sim \begin{cases}
      \mathrm{diag}(Y^{\pm}_{1,s,j},Y^{\pm}_{2,s,j},\dots,Y^{\pm}_{\lfloor\frac{\ell}{2}\rfloor,s,j}) &  \ell\text{ even,}\\
      \mathrm{diag}(Y^{\pm}_{1,s,j},Y^{\pm}_{2,s,j},\dots,Y^{\pm}_{\lfloor\frac{\ell}{2}\rfloor,s,j},\, x\pm \iota\lambda^{m,n}_{s,j}) &  \ell\text{ odd.}\\
  \end{cases}
\end{align*}
Note that 
\[
\det{Y^{+}_{k,s,j}}=\det{Y^{-}_{k,s,j}}={x^2+4a^2 \cos^2{\frac{k\pi}{\ell+1}}+4b^2 \cos^2{\frac{s\pi}{m+1}}+4c^2 \cos^2{\frac{j\pi}{n+1}}}, 
\]
which implies that
\begin{align*}
\det{ F^{+}_{s,j}} \, \det{ F^{-}_{s,j}} &=\prod_{k=1}^{\lfloor \ell/2\rfloor}
(\det{Y^{+}_{k,s,j}}\, \det{Y^{-}_{k,s,j}})  \times
    \begin{cases}
        1 & \ell\text{ even},\\
       x^2 + \left( \lambda^{m,n}_{s,j} \right)^2 & \ell\text{ odd}.
    \end{cases}
\end{align*}
Then we get,
\begin{align*}
\det{F_j}  &=\prod_{s=1}^{\lfloor m/2\rfloor} 
(\det{ F^{+}_{s,j}}\, \det{ F^{-}_{s,j}}) \times \begin{cases}
    1 &  m\text{ even,}\\
    \det{ F^{+}_{\frac{m+1}{2},j}} & m\text{ odd.}
\end{cases}
\end{align*}
Hence,
\begin{multline*}
    \det{F_j}=\prod_{s=1}^{\lfloor m/2\rfloor}\prod_{k=1}^{\lfloor \ell/2\rfloor}\bigg (x^2+4a^2 \cos^2{\frac{k\pi}{\ell+1}}+4b^2 \cos^2{\frac{s\pi}{m+1}}+4c^2 \cos^2{\frac{j\pi}{n+1}}\bigg )^2 \\
\hspace*{-2cm}    \times \begin{cases}
   1 &\ell,m\in 2\mathbb{N}, \\[0.5cm]
\ds \prod_{s=1}^{\lfloor m/2\rfloor} \left( x^2+(\lambda^{m,n}_{s,j})^2 \right)   &\ell\notin 2\mathbb{N},m\in 2\mathbb{N}, \\[0.5cm]
 \ds\prod_{k=1}^{\lfloor \ell/2\rfloor} \big(x^2+4a^2 \cos^2{\frac{k\pi}{\ell+1}}+4c^2 \cos^2{\frac{j\pi}{n+1}}\big)  &\ell\in 2\mathbb{N},m\notin 2\mathbb{N}, \\[0.5cm]
 \ds\prod_{k=1}^{\lfloor \ell/2\rfloor} \big(x^2+4a^2 \cos^2{\frac{k\pi}{\ell+1}}+4c^2 \cos^2{\frac{j\pi}{n+1}}\big) \\[0.5cm]
 \ds\times     \big(x+2\iota c \cos{\frac{j\pi}{n+1}} \big) \prod_{s=1}^{\lfloor m/2\rfloor} \left(x^2+(\lambda^{m,n}_{s,j})^2 \right) &\ell,m\notin 2\mathbb{N}.
   \end{cases}
\end{multline*}
Since, $\det{\mathcal{K}_{\ell,m,n}}=\prod_{j=1}^n\det{F_j}$ and {$\det{F_j}=\det{F_{n-j+1}}$} we obtain the result.
\end{proof}

We now make a few remarks about this result.
First, the orientation on $G$ is Pfaffian over all standard planes and $G$ is non-planar when at least two of $\ell,m,n$ are greater than 2. 
Second, although it is not obvious from \cref{thm:PF3Dgrid},
$\mathcal{Z}_{\ell,m,n}$ is always a polynomial in $x,a,b,c$ with nonnegative integer coefficients.
Third, $\mathcal{Z}_{\ell,m,n}$ is the fourth power of a polynomial when $\ell,m$ and $n$ are all even and the square of a polynomial when exactly two of $\ell,m$ and $n$ are even. 
Fourth, the formula in \cref{thm:PF3Dgrid} coincides with the already known partition function \cite{Ayyer2015ASM} of the monopole-dimer model for the two-dimensional grid graph when either of $\ell,m,n$ are equal to $1$.
Finally, although it is not obvious from the construction, the 
formula is symmetric in all three directions. That is to say, it is symmetric under any permutation interchanging $(a,\ell), (b,m)$ and $(c,n)$.

We now prove that our monopole-dimer model on Cartesian products satisfies an associativity property at least for path graphs.

\begin{proposition}
\label{prop:ASSOCofPlPmPn}
 The partition function of the monopole-dimer model on the oriented Cartesian product of $(P_{\ell}\square P_m,\mathcal{O}_{\ell,m})$ with $(P_n,\mathcal{O}_n)$ is the same as the partition function of the monopole-dimer model on the oriented Cartesian product of $(P_{\ell},\mathcal{O}_{\ell})$ with $(P_m\square P_n,\mathcal{O}_{m,n})$.
\end{proposition}

\begin{proof}
The orientation on both the products is induced from the same boustrophedon labelling given in \cref{Rem:boustro_label}. Moreover, the partition function of the monopole-dimer model is the same as the partition function of the loop-vertex model on $P_{\ell}\square P_m \square P_n$ with the canonical orientation induced from boustrophedon labelling.
\end{proof}

\section{Higher-dimensional grid graphs}
\label{sec:HDgridGraph}

We now generalise the results from \cref{sec:MDM3D} to higher dimensional grid graphs. Let us consider $d$ path graphs $P_{m_{1}}, P_{m_{2}},\dots,P_{m_d}$, $\ell$ of which are odd. Without loss of generality, we assume that the first $\ell$ of these are odd, that is $m_1,\dots, m_\ell$ are odd.
\begin{theorem}
\label{thm:PFkDgrid}
Let $(G,\mathcal{O})$ be the oriented Cartesian product of the graphs $(P_{m_{1}}\square P_{m_{2}},\mathcal{O}_{m_{1},m_{2}}), \allowbreak (P_{m_3}, \mathcal{O}_{m_3}), \allowbreak \dots,(P_{m_d},\mathcal{O}_{m_d})$, where $m_1,\dots, m_\ell$ are odd. Let vertex weights be $x$ for all vertices of $G$ and edge weights be $a_i$ for the $P_{m_i}$-edges. 
Then the partition function of the monopole-dimer model on $G$ is given by
\begin{align}
\label{eqn:PFofkGrid}
    \mathcal{Z}_{G}\equiv\mathcal{Z}_{m_1,\dots,m_d}=
    \prod_{S \subseteq [\ell]} (T_S)^{2^{d-1-\# S}}, 
\end{align}
where for $S=[d]\setminus\{p_1,\dots,p_{r}\}$,
\begin{align*}
T_S = \prod_{i_{p_1}=1}^{\lfloor \frac{m_{p_1}}{2} \rfloor}\cdots \prod_{i_{p_{r}}=1}^{\lfloor \frac{m_{p_{r}}}{2} \rfloor} 
\left( x^2+\sum_{q=1}^{r} 4a_s^2\cos^2{\frac{i_{p_q}\pi }{m_{p_q}+1}} \right),
\end{align*}
and when $\ell = d$, the empty product in $T_{[d]}$ must be interpreted as $x^2$.
\end{theorem}

Note that if $\ell=d$ then the term in \eqref{eqn:PFofkGrid} corresponding to $S = [d]$ is just $x$ which is expected since each configuration will have at least one monomer.
The proof strategy is similar to that of \cite[Section 4]{hammersley-menon-1970}.
Using ideas similar to the proof of \cref{prop:ASSOCofPlPmPn}, it can be shown that for $s \in [d-1]$, the formula above coincides with the partition function of the monopole-dimer model on the oriented Cartesian product $P_{m_1}\square P_{m_2}\square\cdots\square P_{m_{s-1}}\square  (P_{m_{s}}\square P_{m_{s+1}})\square P_{m_{s+2}}\square\cdots\square P_{m_{d}}$.
{We first demonstrate the strategy of proof in an example below}.

\begin{example}
{Consider the $4$-dimensional oriented hypercube, $Q_{4}$, built as an oriented Cartesian product of $4$ copies of $(P_2,\mathcal{O}_2)$  as in \cref{def:ortd Cart}.  }
Then the generalised adjacency matrix is
\begin{align*}
\mathcal{K}_G&=I_{2} \otimes I_{2} \otimes I_{2}\otimes T_{2}(-a_1,x,a_1)
+I_{2}  \otimes I_{2}\otimes T_{2}(-a_2,0,a_2)\otimes J_{2} 
\\
&+I_{2}  \otimes  T_{2}(-a_3,0,a_3)\otimes J_{2}\otimes J_{2}
+T_{2}(-a_4,0,a_4)\otimes J_{2}\otimes J_{2}\otimes J_{2}.
\end{align*}
Let $u_k$ be as defined in the proof of~\cref{thm:PF3Dgrid}. Then, using the unitary transform $u_{2}\otimes u_{2}\otimes u_{2} \otimes u_{2}$,  
we see that 
\begin{align*}
  \mathcal{K}_G &\sim I_{2} \otimes I_{2} \otimes I_{2} \otimes \begin{pmatrix}
   x+\iota a_1& 0\\
    0& x-\iota a_1
  \end{pmatrix} 
  +I_{2} \otimes I_{2} \otimes \begin{pmatrix}
   \iota a_2& 0\\
    0& -\iota a_2
  \end{pmatrix}\otimes \iota\begin{pmatrix}
   0 &-1\\
   1& 0 
  \end{pmatrix}
  \\&+I_{2} \otimes \begin{pmatrix}
   \iota a_3& 0\\
    0& -\iota a_3
  \end{pmatrix}\otimes \iota\begin{pmatrix}
   0 &-1\\
   1& 0 
  \end{pmatrix}\otimes \iota\begin{pmatrix}
   0 &-1\\
   1& 0 
  \end{pmatrix}
  \\&
  +\begin{pmatrix}
   \iota a_4& 0\\
    0& -\iota a_4
  \end{pmatrix}\otimes \iota\begin{pmatrix}
   0 &-1\\
   1& 0 
  \end{pmatrix}\otimes \iota\begin{pmatrix}
   0 &-1\\
   1& 0 
  \end{pmatrix}\otimes \iota\begin{pmatrix}
   0 &-1\\
   1& 0 
  \end{pmatrix}.
\end{align*}
Define, for $1 \leq i_4 \leq 2$,
\begin{align*}
F_{i_4}&=I_{2} \otimes I_{2} \otimes \begin{pmatrix}
   x+\iota a_1& 0\\
    0& x-\iota a_1
  \end{pmatrix} 
  + I_{2} \otimes \begin{pmatrix}
   \iota a_2& 0\\
    0& -\iota a_2
  \end{pmatrix}\otimes \iota\begin{pmatrix}
   0 &-1\\
   1& 0 
  \end{pmatrix}
  \\&+ \begin{pmatrix}
   \iota a_3& (-1)^{i_4-1} a_4\\
    (-1)^{i_4}a_4& -\iota a_3
  \end{pmatrix}\otimes \iota\begin{pmatrix}
   0 &-1\\
   1& 0 
  \end{pmatrix}\otimes \iota\begin{pmatrix}
   0 &-1\\
   1& 0 
  \end{pmatrix},
  \end{align*}
and note that $\det \mathcal{K}_G = \det F_{1} \det F_{2}$.
Now
    \begin{align*}
   F_{1} \sim  F_{2} \sim  &\, I_{2} \otimes I_{2} \otimes \begin{pmatrix}
   x+\iota a_1& 0\\
    0& x-\iota a_1
  \end{pmatrix} 
  + I_{2} \otimes \begin{pmatrix}
   \iota a_2& 0\\
    0& -\iota a_2
  \end{pmatrix}\otimes \iota\begin{pmatrix}
   0 &-1\\
   1& 0 
  \end{pmatrix}
  \\&+ \begin{pmatrix}
   \iota\sqrt{a_3^2+a_4^2}& 0\\
    0& -\iota\sqrt{a_3^2+a_4^2}
  \end{pmatrix}\otimes \iota\begin{pmatrix}
   0 &-1\\
   1& 0 
  \end{pmatrix}\otimes \iota\begin{pmatrix}
   0 &-1\\
   1& 0 
  \end{pmatrix},
 \end{align*}
and thus both $F_1$ and $F_2$ have same the determinant. Hence $\det \mathcal{K}_G=\det F_{1}^2$.
Iterating the same procedure two more times, we get
\[
\det \mathcal{K}_G=\det \begin{pmatrix}
   x+\iota a_1&\sqrt{a_2^2+a_3^2+a_4^2}\\
   -\sqrt{a_2^2+a_3^2+a_4^2}& x-\iota a_1 
  \end{pmatrix}^8.
\]
Thus, the partition function of the monopole-dimer model on $Q_4$ is given by
\[
\mathcal{Z}_{Q_4} = (x^2+a_1^2+a_2^2+a_3^2+a_{4}^2)^8.
\]
\end{example}
\begin{proof}[Proof of \cref{thm:PFkDgrid}]
Using \cref{thm:DetFor2}, the partition function is the determinant of the generalised adjacency matrix, $\mathcal{K}_G$, of $(G,\mathcal{O})$ with the boustrophedon labelling, as discussed in \cref{Rem:boustro_label}. 
It will be convenient for us to index the components in the tensor factors in decreasing order.
Let 
\[
M_j^d= \begin{cases}
I_{m_d} \otimes\dots\otimes I_{m_{2}}\otimes T_{m_1}(-a_1,x,a_1) &  j = 1 \\
I_{m_d} \otimes\dots\otimes I_{m_{j+1}}\otimes T_{m_j}(-a_j,0,a_j) \otimes J_{m_{j-1}}\otimes \dots \otimes J_{m_{1}} & 2\leq j\leq d,
\end{cases}
\]
where $T_k(-s,z,s)$ and $J_k$ are defined in the proof of \cref{thm:PF3Dgrid}.
Then $\mathcal{K}_G$ can be written as
\begin{equation}
    \mathcal{K}_G = M_1^d+\dots +M_d^d.
\end{equation}
For $ j\in [d]$, define the $m_j\times m_j$ diagonal matrix
\begin{equation*}
D_j = {\mathrm{diag} \left(2\iota a_j\cos{\frac{\pi}{m_j+1}},\dots, 2\iota a_j\cos{\frac{m_j\pi}{m_j+1}} \right)}
\end{equation*}
and antidiagonal matrix
\begin{equation*}
J_j^{\prime} = \iota^{m_{j}-1}\begin{pmatrix}
   & &(-1)^{m_{j}-1}\\
     &\iddots &\\
    (-1)^0& & 
  \end{pmatrix}.
\end{equation*}
Let
\begin{align*}
K_j^d= \begin{cases}
I_{m_d} \otimes\dots\otimes I_{m_{2}}\otimes (xI_{m_1}+D_1) &  j = 1 \\
I_{m_d} \otimes\dots\otimes I_{m_{j+1}}\otimes D_j \otimes J_{j-1}^{\prime}\otimes \dots \otimes J_{1}^{\prime} &   2\leq j\leq d.
\end{cases}
\end{align*} 
We again use $\sim$ for the equivalence relation on matrices denoting similarity. Let $u_k$ be as defined in the proof of~\cref{thm:PF3Dgrid}. Then, using the unitary transform $u_{m_d}\otimes\dots \otimes u_{m_1}$,  
we see that $\mathcal{K}_G \sim K_1^d +\cdots+K_{d}^d$.
  Let 
  \[
  \lambda_{i_d, i_{d-1},\dots,i_{p}} =\sqrt{\sum_{s=p}^{d}4a_{s}^2 \cos^2{\frac{i_{s}\pi}{m_{s}+1}}}, 
  1\leq p\leq d.
  \]
Since the first matrix in each tensor product $K_j^d$ is diagonal, $\mathcal{K}_G$ is similar to an $m_d\times m_d$ block diagonal matrix with the block $F_{i_d}$, $i_d\in [m_d]$, given by
\begin{align*}
F_{i_d} &= K_1^{d-1}+\cdots+K_{d-2}^{d-1}+
\left( D_{d-1}+2\iota a_{d}\cos{\frac{i_d\pi}{m_d+1}}J_{d-1}^{\prime} \right) 
\otimes J_{d-2}^{\prime}\otimes \dots
  \otimes 
    J_{1}^{\prime}.
  \end{align*}
Diagonalizing the matrix
\[
D_{d-1}+2\iota a_{d}\cos{\frac{i_d\pi}{m_d+1}}J_{d-1}^{\prime}
\]
leads to the matrix
\[
\begin{cases}
\mathrm{diag}(\iota\lambda_{i_d,1},-\iota \lambda_{i_d,1},\iota\lambda_{i_d,2},-\iota \lambda_{i_d,2},\dots,\iota\lambda_{i_d,\lfloor\frac{m_{d-1}}{2}\rfloor},-\iota \lambda_{i_d,\lfloor\frac{m_{d-1}}{2}\rfloor}) &  
\hspace*{-0.2cm} \text{$m_{d-1}$ even,}\\
\mathrm{diag}(\iota\lambda_{i_d,1},-\iota \lambda_{i_d,1},\iota\lambda_{i_d,2},-\iota \lambda_{i_d,2},\dots,\iota\lambda_{i_d,\lfloor\frac{m_{d-1}}{2} \rfloor},-\iota\lambda_{i_d,\lfloor\frac{m_{d-1}}{2}\rfloor},2a_k \iota \cos{\frac{i_d\pi}{m_d+1}}) & \text{$m_{d-1}$ odd.}
\end{cases}
\]

Set $F_{i_d}^{+}=F_{i_d}$, $F_{i_d}^{-}=F_{m_d-i_d+1}$ for $1\leq i_{d}\leq \lfloor\frac{m_{d}}{2}\rfloor$, and observe that $F_{i_d}^{+}\sim F_{i_d}^{-}$. 
Since the determinant of a matrix is invariant under similarity transformation,
the partition function can now be calculated as
\begin{equation}
\label{eqn:oneProdKD}
\mathcal{Z}_{G}=\prod_{i_d=1}^{\lfloor\frac{m_{d}}{2}\rfloor} (\det{F_{i_d}^{+}})^2\times\begin{cases}
1& \text{$m_{d}$ even,}\\
\det{F_{\frac{m_d+1}{2}}} &\text{$m_{d}$ odd.}
\end{cases}    
\end{equation}

Let us first assume $\ell < d$, i.e. $m_d$ is even.
Repeating the above idea, $F_{i_d}^{+}$ is similar to an $m_{d-1} \times m_{d-1}$ block diagonal matrix with blocks $F^{\pm}_{i_d,i_{d-1}}$ for $1\leq i_{d-1}\leq \lfloor\frac{m_{d-1}}{2}\rfloor$ and continue this process. 
Inductively, we obtain
\begin{equation}
\label{eqn:mlProduct}
 \mathcal{Z}_{G}=\prod_{i_d=1}^{\frac{m_{d}}{2}}\cdots\prod_{i_{\ell+1}=1}^{\frac{m_{\ell+1}}{2}} (\det{F^{+}_{i_d,i_{d-1},\dots,i_{\ell+1}}}\det{F^{-}_{i_d,i_{d-1},\dots,i_{\ell+1}}})^{2^{d-\ell-1}},    
\end{equation}
with 
\begin{align*}
 F^{\pm}_{i_d,i_{d-1},\dots,i_{\ell+1}}&= K_1^{\ell}+\cdots+K_{\ell-1}^{\ell}+ \big(D_{\ell}\pm \iota\lambda_{i_d,\dots,i_{\ell+1}}J_{\ell}^{\prime}\big)\otimes J_{\ell-1}^{\prime}\otimes \dots
  \otimes 
    J_{1}^{\prime}.
\end{align*}
Again diagonalizing, the matrix
\[
D_{\ell}\pm \iota\lambda_{i_d,\dots,i_{\ell+1}}J_{\ell}^{\prime}
\]
is similar to
\[
\mathrm{diag}\Big(\iota\lambda_{i_d,\dots,i_{\ell+1},1},-\iota \lambda_{i_d,\dots,i_{\ell+1},1},
\dots,\iota\lambda_{i_d,\dots,i_{\ell+1},\lfloor\frac{m_{\ell}}{2} \rfloor},-\iota \lambda_{i_d,\dots,i_{\ell+1},\lfloor\frac{m_{\ell}}{2} \rfloor},\pm \iota\lambda_{i_d,\dots,i_{\ell+1},\frac{m_{\ell}+1}{2}}\Big).
\]
Therefore,
\begin{align}
\label{eqn:IntBlocks}
    \det{F^{+}_{i_d,i_{d-1},\dots,i_{\ell+1}}}&=\det{F^{+}_{i_d,\dots,i_{\ell+1},\frac{m_{\ell}+1}{2}}}\prod_{i_{\ell}=1}^{\lfloor m_{\ell}/2\rfloor}(\det{F^{+}_{i_d,\dots,i_{\ell+1},i_{\ell}}}\det{F^{-}_{i_d,\dots,i_{\ell+1},i_{\ell}}}),
\end{align}
and
\begin{align}
\label{eqn:IntBlocks2}
  \det{F^{-}_{i_d,i_{d-1},\dots,i_{\ell+1}}}&=\det{F^{-}_{i_d,\dots,i_{\ell+1},\frac{m_{\ell}+1}{2}}}\prod_{i_{\ell}=1}^{\lfloor m_{\ell}/2\rfloor}(\det{F^{+}_{i_d,\dots,i_{\ell+1},i_{\ell}}}\det{F^{-}_{i_d,\dots,i_{\ell+1},i_{\ell}}}),
\end{align}
where 
\begin{align*}
 F^{\pm}_{i_d,i_{d-1},\dots,i_{\ell}}&= K_1^{\ell-1}+\cdots+K_{\ell-2}^{\ell-1}+ \big(D_{\ell-1}\pm \iota\lambda_{i_d,\dots,i_{\ell}}J_{\ell-1}^{\prime}\big)\otimes J_{\ell-2}^{\prime}\otimes \dots
  \otimes 
    J_{1}^{\prime}.
\end{align*}
Substituting \eqref{eqn:IntBlocks} and \eqref{eqn:IntBlocks2} in \eqref{eqn:mlProduct} gives
\begin{multline*}
 \mathcal{Z}_{G}=\prod_{i_d=1}^{\frac{m_{d}}{2}}\cdots\prod_{i_{\ell}=1}^{\frac{m_{\ell}}{2}} (\det{F^{+}_{i_d,i_{d-1},\dots,i_{\ell}}}\det{F^{-}_{i_d,i_{d-1},\dots,i_{\ell}}})^{2^{d-\ell}}\\
\prod_{i_d=1}^{\frac{m_{d}}{2}}\cdots\prod_{i_{\ell+1}=1}^{\frac{m_{\ell+1}}{2}} \Big(\det{F^{+}_{i_d,\dots,i_{\ell+1},\frac{m_{\ell}+1}{2}}}\det{F^{-}_{i_d,\dots,i_{\ell+1},\frac{m_{\ell}+1}{2}}}\Big)^{2^{d-\ell-1}}    
\end{multline*}
By repeated application of this procedure, we will get 
\begin{equation*}
\mathcal{Z}_{G}=\prod_{i_d=1}^{\frac{m_{d}}{2}}\dots \prod_{i_1=1}^{\lfloor\frac{m_{1}}{2}\rfloor} (\det{F^{+}_{i_d,\dots,i_1}}\det{F^{-}_{i_d,\dots,i_1}})^{2^{d-1}}    \prod_{\substack{S\subset [\ell]\\|S|=1}}T_{S}^{2^{d-2}}\prod_{\substack{S\subset [\ell]\\|S|=2}}T_{S}^{2^{d-3}}\cdots \prod_{\substack{S\subset [\ell]\\|S|=\ell}}T_{S}^{2^{d-\ell-1}},
\end{equation*}
where $F^{\pm}_{i_d,\dots,i_1}$ is the $1 \times 1$ matrix
\[
\begin{pmatrix}
  \ds \pm \iota \sqrt{x^2+\sum_{s=1}^{d}4a_{s}^2 \cos^2{\frac{i_{s}\pi}{m_{s}+1}}}
\end{pmatrix}.
\]
Thus, we finally arrive at 
\begin{equation}
\label{eqn:PFwhnlneqd}
\mathcal{Z}_{G}=\prod_{i_d=1}^{\frac{m_{d}}{2}}\dots \prod_{i_1=1}^{\lfloor\frac{m_{1}}{2}\rfloor} \Big(x^2+\sum_{s=1}^{d}4a_{s}^2 \cos^2{\frac{i_{s}\pi}{m_{s}+1}}\Big)^{2^{d-1}}    \prod_{\substack{S\subset [\ell]\\|S|=1}}T_{S}^{2^{d-2}}\prod_{\substack{S\subset [\ell]\\|S|=2}}T_{S}^{2^{d-3}}\cdots \prod_{\substack{S\subset [\ell]\\|S|=\ell}}T_{S}^{2^{d-\ell-1}}.
\end{equation}
The last case is when $\ell=d$. Then the right hand side of~\eqref{eqn:PFwhnlneqd} will have an additional factor of $\det{F_{\frac{m_d+1}{2}}}$, and the latter is the generalised adjacency matrix for the oriented Cartesian product $(P_{m_{1}}\square P_{m_{2}},\mathcal{O}_{m_{1},m_{2}}) \square \allowbreak (P_{m_3}, \mathcal{O}_{m_3}) \square  \cdots \square (P_{m_{d-1}},\mathcal{O}_{m_{d-1}})$. Thus, 
 \begin{align*}
\det{F_{\frac{m_d+1}{2}}} =  \prod_{S \subseteq [d-1]} (T_S)^{2^{d-2-\# S}} 
= \prod_{\substack{\emptyset \subsetneq S \subseteq [d]\\d\in S}} (T_S)^{2^{d-1-\# S}},
 \end{align*}
by induction.
Substituting this in \eqref{eqn:oneProdKD} completes the proof.
\end{proof}

As for the three-dimensional case, it is not obvious from the formula \eqref{eqn:PFofkGrid} for $\mathcal{Z}_{m_1,\dots,m_d}$ that it is a polynomial with nonnegative integer coefficients.
The formula is also symmetric under any permutation of $(a_1,m_1), \dots, (a_k,m_d)$.
Finally, \eqref{eqn:PFofkGrid} tells that the partition function of the monopole-dimer model for even grid lengths is the $2^{(d-1)}$'th power of a polynomial.
Again a combinatorial interpretation of the underlying polynomial would be interesting. 

We end this section with an example for a well-studied family of graphs.

\begin{example}
Consider the $d$-dimensional oriented hypercube, $Q_{d}$, built as an oriented Cartesian product of $d$ copies of $(P_2,\mathcal{O}_2)$  as in \cref{def:ortd Cart}.  
Then the partition function of the monopole-dimer model on $Q_d$ is given by
\[
\mathcal{Z}_{Q_d} = (x^2+a_1^2+\cdots+a_{d}^2)^{2^{d-1}}.
\]
{While this formula is amazingly simple, it is a result of a lot of cancellation of terms. Finding a combinatorial interpretation of this formula would certainly be very interesting. An interpretation in the two-dimensional case has been given in \cite{ayyer-2020}.}
\end{example}

\section{Asymptotic Behaviour}
\label{sec:asym}

{It is natural to ask how fast the partition function of these monopole-dimer models on grid graphs grows as the size increases. We are also interested in understanding the `probability' of seeing a monopole at a given vertex or a dimer at a given edge. The reason these are not strict probabilities is that we are working with signed measures.
As a warm-up, we begin with three-dimensional grids in \cref{sec:asym-3d}. We then move on the general $d$-dimensional grids in \cref{sec:asym-gend}, where the formulas are not as explicit.}
We will follow the strategy in~\cite[Section 5]{Ayyer2015ASM}. 

\subsection{Three-dimensional grids}
\label{sec:asym-3d}

Define the \emph{free energy}
as
\begin{equation*}
    \Phi_3(a,b,c,x):=\lim_{\ell,m,n\rightarrow \infty}\frac{1}{8\ell mn}\ln{\,\mathcal{Z}_{2\ell,2m,2n}}.
\end{equation*}
Using the product formula in \cref{thm:PF3Dgrid},
\begin{multline*}
\Phi_3(a,b,c,x)=\lim_{\ell,m,n\rightarrow \infty}\frac{1}{8\ell mn} \\
\times \sum\limits_{j=0}^{n}\sum\limits_{s=0}^{m} \sum\limits_{k=0}^{\ell} \ln\bigg(x^2+4a^2\cos^2{\frac{\pi k}{2\ell+1}}+4b^2\cos^2{\frac{\pi s}{2m+1}}+4c^2\cos^2{\frac{\pi j}{2n+1}}\bigg )^4.
\end{multline*}
Note that the right hand side can be expressed as a Riemann sum.
Therefore,
\begin{equation*}
    \Phi_3(a,b,c,x)=\frac{4}{\pi^3}\int\limits_0^{\pi/2}\int\limits_0^{\pi/2} \int\limits_0^{\pi/2} \ln({x^2+4a^2\cos^2{\theta}+4b^2\cos^2{\phi}+4c^2\cos^2{\psi})} \, \text{d}{\theta} \, \text{d}{\phi} \, \text{d}{\psi}.
\end{equation*}
Hence, the \emph{density} of $a$-type edges and of monopoles will be
\begin{align*}
    \rho_{3,a}&:=a\frac{\partial}{\partial a} \Phi_3=\frac{4}{\pi^3}\int\limits_0^{\pi/2}\int\limits_0^{\pi/2} \int\limits_0^{\pi/2} \frac{8a^2\cos^2{\theta}}{({x^2+4a^2\cos^2{\theta}+4b^2\cos^2{\phi}+4c^2\cos^2{\psi})}} \, \text{d}{\theta} \, \text{d}{\phi} \, \text{d}{\psi}, \\
    \rho_{3,x}&:=x\frac{\partial}{\partial x} \Phi_3=\frac{4}{\pi^3}\int\limits_0^{\pi/2}\int\limits_0^{\pi/2} \int\limits_0^{\pi/2} \frac{2x^2}{({x^2+4a^2\cos^2{\theta}+4b^2\cos^2{\phi}+4c^2\cos^2{\psi})}} \, \text{d}{\theta} \, \text{d}{\phi}\, \text{d}{\psi},
\end{align*}
respectively.
Similarly, the density of $b$- and $c$-type dimers can be defined and 
one can check that $\rho_{3,a}+\rho_{3,b}+\rho_{3,c}+\rho_{3,x}=1$ as expected.

Recall the \emph{elliptic integral of the first kind}, 
$$
F(\phi,k)=\int\limits_{0}^{\phi} \frac{\text{d}\alpha}{\sqrt{1-k^2\sin^2{\alpha}}},
$$
and the \emph{elliptic integral of the second kind}, 
$$
E(\phi,k)=\int\limits_{0}^{\phi} \sqrt{1-k^2\sin^2{\alpha}}\,\text{d}\alpha.
$$ 
The \emph{complete elliptic integral of the first kind} is $K(k) = F(\pi/2,k)$ and the \emph{complete elliptic integral of the second kind} is
 $E(k) = E(\pi/2,k)$. Then, the \emph{Jacobi zeta function} is 
$$
Z(\phi,k)=E(\phi,k)-\frac{E(k)}{K(k)}F(\phi,k).
$$
See Gradshteyn and Rizhik~\cite{TableOfIntegrals2000} for basic properties of elliptic integrals.

Now performing similar calculations as in~\cite{Ayyer2015ASM} using
\begin{align*}
\epsilon_3 =\tan^{-1}{\bigg(\frac{\sqrt{x^2+4c^2\cos^2{\psi}+4b^2}}{2a}\bigg)}
\end{align*}
and
\begin{align*}
q_3 =\frac{4ab}{\sqrt{(x^2+4c^2\cos^2{\psi}+4a^2)(x^2+4c^2\cos^2{\psi}+4b^2)}},
\end{align*}    
we get,
\begin{align}
    \rho_{3,a} =& 1-\frac{2}{\pi}\int \limits_0^{\pi/2}\Lambda_{0}(\epsilon_3, \sin^{-1}{q_3}) \, d\psi, \\
\label{eqn:Densities3D}
    \rho_{3,x} =& \frac{x^2}{\pi^2 ab}\int \limits_0^{\pi/2}q_3\,K(q_3) \, \text{d}\psi,
\end{align}
where $\Lambda_{0}(\theta,y)$ is the \emph{Heuman lambda function}~\cite[~Formula 17.4.39]{AandS-1964} defined as
\begin{equation*}
    \Lambda_{0}(\theta,y)
    =\frac{F(\phi, \cos{y})}{K(\cos{y})}+\frac{2}{\pi}K(\sin{y})Z(\phi,\cos{y}).
\end{equation*}
Let us now calculate the monopole density for the three dimensional case when all the vertex and edge weights are $1$. Using \eqref{eqn:Densities3D}, we get
\[
\rho_{3,x} = \frac{1}{\pi^2}\int \limits_0^{\pi/2}\frac{4}{5+4\cos^2{\psi}}\,K\Big(\frac{4}{5+4\cos^2{\psi}}\Big) \, \text{d}\psi
\approx 0.1705.
\]

\subsection{$d$-dimensional grids}
\label{sec:asym-gend}

We now move on to the case of $d$-dimensional grid graphs where all side lengths are even, vertex weights are $x$ and edges along the $j$'th direction have weight $a_j$. The free energy is given by
\begin{equation*}
    \Phi_d(a_1,\dots,a_d,x):=\lim_{m_1,\dots,m_d\rightarrow \infty}\frac{1}{2^dm_1\cdots m_d}\ln{\,\mathcal{Z}_{2m_1,\dots,2m_d}}.
\end{equation*}
The product formula in \cref{thm:PFkDgrid} together with the Riemann sum implies that 
\begin{equation*}
    \Phi_d(a_1,\dots,a_d,x)=\frac{2^{d-1}}{\pi^d}\int\limits_0^{\pi/2}\cdots\int\limits_0^{\pi/2} \ln\bigg({x^2+\sum_{s=1}^d 4a_s^2\cos^2{\theta_s}\bigg)} \,\text{d}{\theta_{1}} \dots \text{d}{\theta_{d}}.
\end{equation*}
Again defining the densities of monopoles and $s$-type edges for $s \in [d]$ as
\[
\rho_{d,x} :=x\frac{\partial}{\partial x} \Phi_d, \quad 
\rho_{d,a_s} := a_s \frac{\partial}{\partial a_s} \Phi_d,
\]
one can again get that $\rho_{d,x} + \sum_{s=1}^d\rho_{d,a_s}=1$.
Following the strategy in~\cite{Ayyer2015ASM}, we define
\begin{align*}
    \epsilon_d &=\tan^{-1} \Bigg(\frac{\sqrt{x^2 +4a_2^2 +\sum_{s=3}^d 4a_s^2\cos^2{\theta_s}}}{2a_1}\Bigg), \\
    q_d &=\frac{4a_1a_2}{
    \sqrt{(x^2 +4a_1^2 +\sum_{s=3}^d 4a_s^2\cos^2{\theta_s})
    (x^2 +4a_2^2 +\sum_{s=3}^d 4a_s^2\cos^2{\theta_s})}},
    \end{align*}
and we get
\begin{equation*}
\begin{split}
    \rho_{d,a_1} =& 1-\frac{2^{d-2}}{\pi^{d-2}}\int\limits_0^{\pi/2}\cdots\int\limits_0^{\pi/2}\Lambda_{0}(\epsilon_d, \sin^{-1}{q_d})\, \text{d}{\theta_{3}} \dots \text{d}{\theta_{d}}, \\
    \rho_{d,x} =& \frac{2^{d-3}x^2}{\pi^{d-1} a_1a_2}\int\limits_0^{\pi/2}\cdots\int\limits_0^{\pi/2}q_d \,K(q_d)\, \text{d}{\theta_{3}} \dots \text{d}{\theta_{d}}.
\end{split}
\end{equation*}

\begin{figure}[h!]
    \centering
    \includegraphics[scale=0.8]{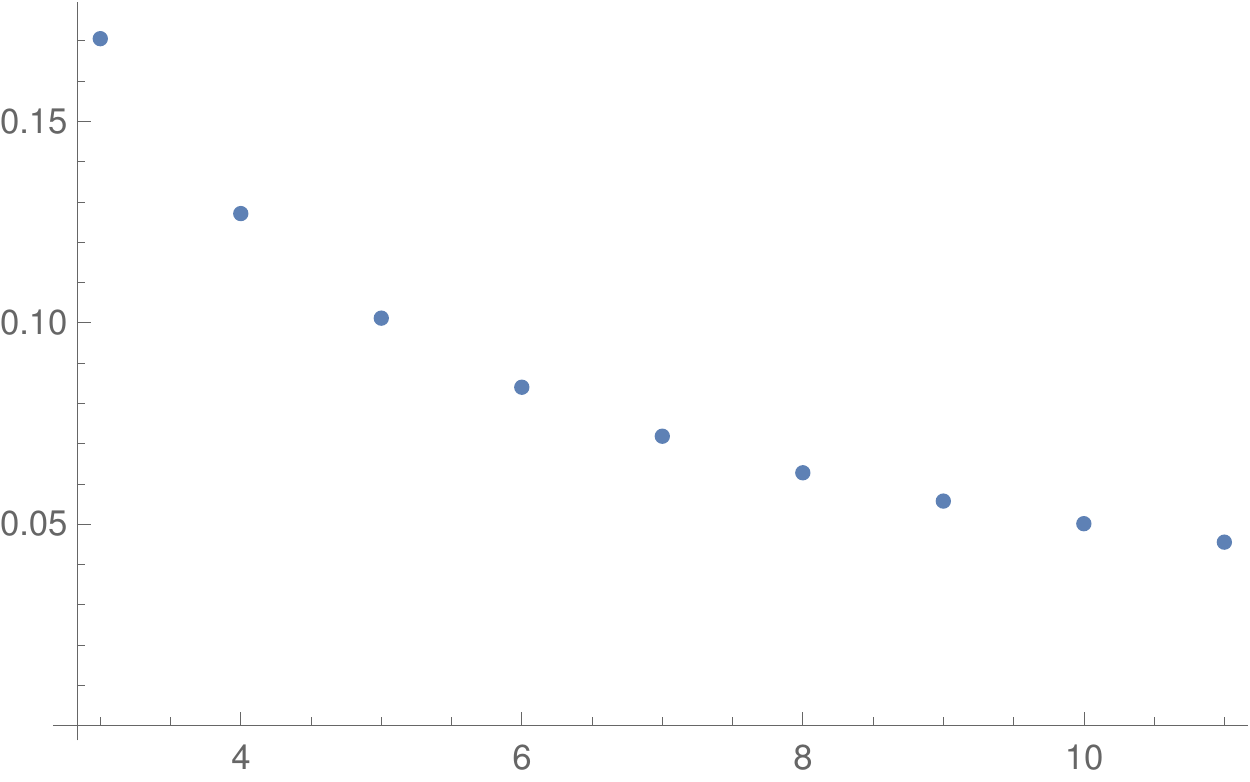}
    \caption{Monopole densities $\rho_{d,x}$ for limiting grid graphs in dimensions $d$ ranging from 3 to 11 when all the vertex and edge weights are $1$.}
    \label{fig:MonopoledensitiesD3to11}
\end{figure}

\cref{fig:MonopoledensitiesD3to11} shows the numerically evaluated monopole density $\rho_{d,x}$ for the first few dimensions when all the vertex and edge weights are $1$.
Observe that $\rho_{d,x}$ seems to decrease monotonically as dimension increases. It would be interesting to determine the limit of $\rho_{d,x}$ as $d$ tends to infinity, if it exists. In particular, it is not clear whether this limit is 0 or not.

\section*{Acknowledgements}
We thank the referees for many helpful comments and suggestions.

\bibliographystyle{alphaurl}
\bibliography{bib}

\begin{thebibliography}{CSW23}

\bibitem[AS64]{AandS-1964}
Milton Abramowitz and Irene~A Stegun.
\newblock {\em Handbook of mathematical functions with formulas, graphs, and
  mathematical tables}, volume~55.
\newblock US Government printing office, 1964.

\bibitem[Ayy15]{Ayyer2015ASM}
Arvind Ayyer.
\newblock A statistical model of current loops and magnetic monopoles.
\newblock {\em Math. Phys. Anal. Geom.}, 18(1):Art. 16, 19, 2015.
\newblock URL: \url{https://doi.org/10.1007/s11040-015-9185-6}, \href
  {http://dx.doi.org/10.1007/s11040-015-9185-6}
  {\path{doi:10.1007/s11040-015-9185-6}}.

\bibitem[Ayy20]{ayyer-2020}
Arvind Ayyer.
\newblock Squareness for the monopole-dimer model.
\newblock {\em Ann. Comb.}, 24(2):237--255, 2020.
\newblock URL: \url{https://doi.org/10.1007/s00026-019-00480-5}, \href
  {http://dx.doi.org/10.1007/s00026-019-00480-5}
  {\path{doi:10.1007/s00026-019-00480-5}}.

\bibitem[BM08]{BondyGT}
J.~A. Bondy and U.~S.~R. Murty.
\newblock {\em Graph theory}, volume 244 of {\em Graduate Texts in
  Mathematics}.
\newblock Springer, New York, 2008.
\newblock URL: \url{https://doi.org/10.1007/978-1-84628-970-5}, \href
  {http://dx.doi.org/10.1007/978-1-84628-970-5}
  {\path{doi:10.1007/978-1-84628-970-5}}.

\bibitem[Cay49]{Cayley1849}
A.~Cayley.
\newblock Sur les déterminants gauches. (suite du mémoire t. xxxii. p. 119).
\newblock {\em Crelle's journal}, 1849(38):93--96, 1849.
\newblock URL: \url{https://doi.org/10.1515/crll.1849.38.93}, \href
  {http://dx.doi.org/doi:10.1515/crll.1849.38.93}
  {\path{doi:doi:10.1515/crll.1849.38.93}}.

\bibitem[CSW23]{CSW-2023}
Nishant Chandgotia, Scott Sheffield, and Catherine Wolfram.
\newblock Large deviations for the {3D} dimer model.
\newblock 2023.
\newblock \href {http://arxiv.org/abs/2304.08468} {\path{arXiv:2304.08468}}.

\bibitem[Fis61]{Fisher}
Michael~E. Fisher.
\newblock Statistical mechanics of dimers on a plane lattice.
\newblock {\em Phys. Rev.}, 124:1664--1672, Dec 1961.
\newblock URL: \url{https://link.aps.org/doi/10.1103/PhysRev.124.1664}, \href
  {http://dx.doi.org/10.1103/PhysRev.124.1664}
  {\path{doi:10.1103/PhysRev.124.1664}}.

\bibitem[GR00]{TableOfIntegrals2000}
I.~S. Gradshteyn and I.~M. Ryzhik.
\newblock {\em Table of integrals, series, and products}.
\newblock Academic Press Inc., San Diego, CA, sixth edition, 2000.
\newblock Translated from the Russian, Translation edited and with a preface by
  Alan Jeffrey and Daniel Zwillinger.

\bibitem[HLT23]{HLT-2023}
Ivailo Hartarsky, Lyuben Lichev, and Fabio Toninelli.
\newblock Local dimer dynamics in higher dimensions, 2023.
\newblock \href {http://arxiv.org/abs/2304.10930} {\path{arXiv:2304.10930}}.

\bibitem[HM70]{hammersley-menon-1970}
J.~M. Hammersley and V.~V. Menon.
\newblock A lower bound for the monomer-dimer problem.
\newblock {\em J. Inst. Math. Appl.}, 6:341--364, 1970.
\newblock \href {http://dx.doi.org/10.1093/imamat/6.4.341}
  {\path{doi:10.1093/imamat/6.4.341}}.

\bibitem[Jer87]{jerrum1987}
Mark Jerrum.
\newblock Two-dimensional monomer-dimer systems are computationally
  intractable.
\newblock {\em Journal of Statistical Physics}, 48(1):121--134, 1987.
\newblock URL: \url{https://doi.org/10.1007/BF01010403}, \href
  {http://dx.doi.org/10.1007/BF01010403} {\path{doi:10.1007/BF01010403}}.

\bibitem[Kas61]{KASTELEYN19611209}
P.W. Kasteleyn.
\newblock The statistics of dimers on a lattice: I. the number of dimer
  arrangements on a quadratic lattice.
\newblock {\em Physica}, 27(12):1209--1225, 1961.
\newblock URL:
  \url{https://www.sciencedirect.com/science/article/pii/0031891461900635},
  \href {http://dx.doi.org/https://doi.org/10.1016/0031-8914(61)90063-5}
  {\path{doi:https://doi.org/10.1016/0031-8914(61)90063-5}}.

\bibitem[Kas63]{Kasteleyn1963}
P.~W. Kasteleyn.
\newblock Dimer statistics and phase transitions.
\newblock {\em Journal of Mathematical Physics}, 4(2):287--293, 1963.
\newblock URL: \url{https://doi.org/10.1063/1.1703953}, \href
  {http://dx.doi.org/10.1063/1.1703953} {\path{doi:10.1063/1.1703953}}.

\bibitem[TF61]{TemperleyFisher}
H.~N.~V. Temperley and Michael~E. Fisher.
\newblock Dimer problem in statistical mechanics-an exact result.
\newblock {\em The Philosophical Magazine: A Journal of Theoretical
  Experimental and Applied Physics}, 6(68):1061--1063, 1961.
\newblock URL: \url{https://doi.org/10.1080/14786436108243366}, \href
  {http://dx.doi.org/10.1080/14786436108243366}
  {\path{doi:10.1080/14786436108243366}}.

\bibitem[Wu06]{wu2006}
F.~Wu.
\newblock Pfaffian solution of a dimer-monomer problem: Single monomer on the
  boundary.
\newblock {\em Physical Review E}, 74(2), 2006.
\newblock \href {http://dx.doi.org/10.1103/PhysRevE.74.020104}
  {\path{doi:10.1103/PhysRevE.74.020104}}.

\end{thebibliography}

\end{document}